\documentclass[11pt]{article}

\usepackage{authblk}  
\usepackage{fullpage} 
\usepackage{graphicx}
\usepackage{amsfonts,amsmath,amssymb,amsbsy,amsthm,enumitem}
\setlist[description]{leftmargin=3mm,labelindent=0mm}
\usepackage[utf8]{inputenc}
\usepackage[english]{babel}
\usepackage{verbatim}
\usepackage{tikz}
\usepackage{subcaption}
\usetikzlibrary{calc,positioning,decorations.pathmorphing,decorations.pathreplacing}
\usepackage{./Figures/figstyle}
\newcommand{\calV}{\mathcal{V}}
\newcommand{\calP}{\mathcal{P}}

\newcommand{\tw}{\mathrm{tw}}

\usepackage[normalem]{ulem}

\newtheorem{theorem}{Theorem}[section]

\newtheorem{conjecture}[theorem]{Conjecture}
\newtheorem{proposition}[theorem]{Proposition}
\newtheorem{lemma}[theorem]{Lemma}
\newtheorem{corollary}[theorem]{Corollary}
\newtheorem*{definition}{Definition}
\newtheorem{remark}{Remark}
\newtheorem{claim}{Claim}
\newtheorem*{claimsn}{Claim}

\usepackage{marginnote}

\title{On Tuza's conjecture for triangulations \\ and graphs with small treewidth} 

\author[1]{Fábio Botler}
\author[2]{Cristina G. Fernandes}
\author[2]{Juan Guti\'errez}
\affil[1]{Programa de Engenharia de Sistemas e Computação, COPPE,
Universidade~Federal~do~Rio~de~Janeiro,~Brazil}
\affil[2]{Departamento de Ci\^encia da Computa\c c\~ao, Universidade de  S\~{a}o Paulo, Brazil}
\begin{document}
\maketitle

\begin{abstract} 
	Tuza (1981)\nocite{Tuza81} conjectured that the size~\(\tau(G)\) of a minimum set of
	edges that intersects every triangle of a graph~\(G\) is at most twice the size~\(\nu(G)\) 
	of a maximum set of edge-disjoint triangles of~\(G\).
	In this paper we present three results regarding Tuza's Conjecture.
	We verify it for graphs with treewidth at most~\(6\);
	we show that \(\tau(G)\leq \frac{3}{2}\,\nu(G)\) for every planar triangulation \(G\) different from $K_4$;
	and that \(\tau(G)\leq\frac{9}{5}\,\nu(G) + \frac{1}{5}\) if \(G\) is a maximal graph with treewidth 3.
        Our first result strengthens a result of Tuza, implying that \(\tau(G) \leq 2\,\nu(G)\) for every \(K_8\)-free chordal graph \(G\).
\end{abstract}

\section{Introduction}

In this paper all graphs considered are simple and the notation and terminology are standard.
A \emph{triangle transversal} of a graph~\(G\) is a set 
of edges of~\(G\) whose removal results in a triangle-free graph;
and a \emph{triangle packing} of~\(G\) is a set 
of edge-disjoint triangles of~\(G\).
We denote by~\(\tau(G)\) (resp.~\(\nu(G)\)) the size of a minimum triangle transversal 
(resp.~maximum triangle packing) of~\(G\).
Tuza~\cite{Tuza81} posed the following conjecture.

\begin{conjecture}[Tuza, 1981]\label{conjecture:tuza}
  For every graph~\(G\), we have~\(\tau(G)\leq 2\,\nu(G)\).
\end{conjecture}%

This conjecture was verified for many classes of graphs.
In particular, Tuza~\cite{Tuza90} verified it for planar graphs, 
and Haxell and Kohayakawa~\cite{Haxell98} proved that if \(G\) is a tripartite graph, then \(\tau(G)\leq 1.956\,\nu(G)\).
The reader may refer to~\cite{BaronKahn16,ChenDiaoHuTang16,Cui09,Haxell99,Haxell12,Krivelevich95} for more results concerning Tuza's Conjecture.
In this paper we present three results regarding Tuza's Conjecture.
We verify it for graphs with treewidth at most~\(6\);
and we show that \(\tau(G)\leq \frac{3}{2}\,\nu(G)\) for every planar triangulation~\(G\) different from $K_4$;
and that \(\tau(G)\leq\frac{9}{5}\,\nu(G) + \frac{1}{5}\) if \(G\) is a \emph{\(3\)-tree},
i.e., a graph obtained from \(K_3\) by successively choosing a triangle in the graph and adding a new vertex adjacent to its three vertices.

Puleo~\cite{Puleo15} introduced a set of tools for dealing with graphs that contain vertices of small degree (Lemma~\ref{lemma:puleo}),
and verified Tuza's Conjecture for graphs with maximum average degree less than~\(7\),
i.e., for graphs in which every subgraph has average degree less than \(7\). 
In this paper, we extend Puleo's technique (Lemma~\ref{lemma:tuza-tw6-main}) 
in order to prove Tuza's Conjecture for graphs with treewidth at most~\(6\) (Theorem~\ref{thm:tuza-tw-6}).
Note that there are graphs with treewidth at most~\(6\) and 
maximum average degree at least~\(7\) (Figure~\ref{fig:tuza-partial6-tree-deg>7}).
In particular, any graph with treewidth at most~\(6\) that contains such a graph 
also has maximum average degree at least~\(7\).
In particular, this result strengthens a result of Tuza, implying that \(\tau(G) \leq 2\,\nu(G)\) for every \(K_8\)-free chordal graph \(G\).

\begin{figure}[h]
  \centering
  \begin{tikzpicture}[scale = .8]

  \foreach \i in {1,...,3}{
    \node(x\i) [black vertex] at (-30+\i*120:2) {};
  }	
  \foreach \i in {0,...,9}{
    \node(y\i) [black vertex] at (\i*60:1) {};	
  }
  
  \foreach \i in {1,...,5}{
    \foreach \j in {\i,...,6}{
      \draw[edge] (y\i) -- (y\j);
    }
  }
  
  \foreach \i in {1,...,3}{
    {\pgfmathtruncatemacro{\j}{2*\i - 0}
      \draw[edge] (x\i) to [bend left=0] (y\j);}
    {\pgfmathtruncatemacro{\j}{2*\i - 2}
      \draw[edge] (x\i) to [bend left=15] (y\j);}
    {\pgfmathtruncatemacro{\j}{2*\i - 1}
      \draw[edge] (x\i) to [bend left=0] (y\j);}
    {\pgfmathtruncatemacro{\j}{2*\i + 3}
      \draw[edge] (x\i) to [bend right=0] (y\j);}
    {\pgfmathtruncatemacro{\j}{2*\i + 1}
      \draw[edge] (x\i) to [bend right=15] (y\j);}
    {\pgfmathtruncatemacro{\j}{2*\i + 2}
      \draw[edge] (x\i) to [bend right=0] (y\j);}
  }

\end{tikzpicture}
  \caption{A graph with treewidth $6$ and average degree $22/3$.}
  \label{fig:tuza-partial6-tree-deg>7}
\end{figure}
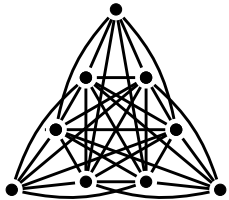

In another direction, motivated by~\cite{Haxell98}, we show that,
for certain classes of graphs, the ratio \(\tau(G)/\nu(G)\) can be bounded by a constant smaller than \(2\).
More specifically, we show that, if \(G\) is a planar triangulation different from $K_4$, then \(\tau(G) \leq \frac32\,\nu(G)\) (Theorem~\ref{thm:tuza-triangulations}) and, if \(G\) is a \(3\)-tree, then \(\tau(G)\leq \frac{9}{5}\,\nu(G) + \frac{1}{5}\) (Theorem~\ref{thm:tuza-3-trees}).

This paper is organized as follows.
In Section~\ref{section:tree-decompositions}, we establish the notation and present some auxiliary results used throughout the paper.
In Section~\ref{section:tuza-partial-6-trees}, we verify Tuza's Conjecture for graphs with treewidth at most~\(6\)
and, in Sections~\ref{section:tuza-planar-triangulation} and~\ref{section:tuza-3-trees}, 
we study planar triangulations and \(3\)-trees, respectively.
Finally, in Section~\ref{section:conclusion} we present some concluding remarks.

\section{Rooted tree decompositions}\label{section:tree-decompositions}

In this section we present most of the notation and some auxiliary results regarding tree decompositions.
(See~\cite[Chp.~12]{Diestel10} for an overview on this concept.)
A \emph{rooted tree} is a pair~$(T,r)$, where~$T$ is a tree and~$r$ is a vertex of~$T$.
Given~${t \in V(T)}$, if~\(t'\) is a vertex in the (unique) path in~\(T\) that joins~\(r\) and~\(t\), 
then we say that~\(t'\) is an \emph{ancestor} of~\(t\).
Every vertex in~$T$ that has~$t$ as its ancestor is called a \emph{descendant} of~$t$.
If~$t \neq r$, then the \emph{parent} of~$t$, denoted by~$p(t)$,
is the ancestor of~$t$ that is adjacent to~$t$.
The \emph{successors} of~$t$ are the vertices whose parent is~$t$,
and we denote the set of successors of~\(t\) by~$S_T(t)$.
The \emph{height} of~$t$, denoted by~$h_T(t)$, is the length (number of edges)
of a longest path that joins~$t$ to a descendant of~$t$.
When~$T$ is clear from the context, we simply write~$S(t)$ and~$h(t)$.

\begin{definition}
A \emph{tree decomposition}  of a graph~$G$
is a pair~$\mathcal{D}=(T, \calV)$ consisting of a tree~$T$ and 
a collection $\calV = \{V_t\subseteq V(G)\colon t \in V(T)\}$, 
satisfying the following conditions: 
\begin{enumerate}[label = {\rm (T\arabic*)}]
\item $\bigcup_{t\in V(T)}V_t = V(G);$ 
\item for every~$uv \in E(G)$, there exists a $t$ such that~$u,v \in V_t;$
\item if a vertex~$v \in V_{t_1} \cap V_{t_2}$ for $t_1 \neq t_2$, 
then \(v\in V_t\) for every \(t\) in the path of~\(T\) that joins~\(t_1\) and~\(t_2\).
In other words, for any fixed vertex \(v\in V(G)\), 
the subgraph of \(T\) induced by the vertices in sets $V_t$ that contain \(v\) is connected.
\end{enumerate}
The elements in \(\calV\) are called the \emph{bags} of \(\mathcal{D}\),
and the vertices of \(T\) are called \emph{nodes}.
The \emph{width} of~$\mathcal{D}$ is the number 
$\max\{|V_t|\colon t\in V(T)\}-1$,
and the \emph{treewidth}~$\tw(G)$ of~$G$ is 
the width of a tree decomposition of~$G$ with minimum width.
Let~$G$ be a graph with treewidth~$k$.
If~$|V_t|=k+1$ for every~$t\in V(T)$,
and~$|V_t \cap V_{t'}|=k$ for every~$tt' \in E(T)$,
then we say that~$(T, \mathcal{V})$ is a \emph{full tree decomposition} of~$G$.
\end{definition}

Note that 3-trees are maximal graphs with treewidth~3.  
Indeed, the construction of a 3-tree~\(G\) describes a tree decomposition of \(G\) whose bags are exactly the \(K_4\)'s formed by the addition of each new vertex.

The following result was proved by Bodlaender~\cite{Bodlaender98} (see also Gross~\cite{Gross14}).
\begin{proposition}\label{prop:full-tree-decomposition}
  Every graph admits a full tree decomposition.
\end{proposition}

A triple~$(\calV,T,r)$ is a \emph{rooted tree decomposition} of a graph~$G$ 
if~$(\calV,T)$ is a full tree decomposition of~$G$, $(T,r)$ is a rooted tree, 
and~$V_t \cap V_{p(t)} \neq V_t \cap V_{t'}$ 
for every~${t \in V(T) \setminus \{r\}}$ and~$t' \in S(t)$.
In what follows, we show that 
every full tree decomposition can be modified into 
a rooted tree decomposition with an arbitrary root $r$.

\begin{proposition}\label{prop:root-tree-dec}
  Every graph admits a rooted tree decomposition.
\end{proposition}

\begin{proof}
  Let~\((T,\calV,r)\) be a triple where~$(T,\calV)$ is a full tree decomposition of~$G$ 
  and~\((T,r)\) is a rooted tree, with $r$ chosen arbitrarily in~$T$.
  Let~\(P_T(t)\) be the (unique) path in~\(T\) that joins~\(r\) and~\(t\).
  Choose such \((T,\calV,r)\) with $\sum_{t \in V(T)}{|P_T(t)|}$ minimum.
  We claim that~$(\calV,T,r)$ is a rooted tree decomposition.
  Suppose, for a contradiction, that there exist two nodes~${t \in V(T)} \setminus \{r\}$ 
  and~${t' \in S(t)}$ such that~${V_t \cap V_{p(t)} = V_t \cap V_{t'}}$.
  Let~$T'$ be the tree obtained from \(T\) by removing the edge \(tt'\) and 
  adding the edge \(p(t)t'\), that is, \(T'\) is such that~$V(T)=V(T')$
  and~${E(T')=E(T) \setminus \{tt'\} \cup \{p(t)t'\}}$.
  Note that~$(T',\calV)$ is a full tree decomposition of~$G$, and
  that~$|P_{T'}(t'')|=|P_{T}(t'')|-1$ for every descendant~$t''$ of~$t'$.
  Hence~${\sum_{t \in V(T')}{|P_{T'}(t)|} < \sum_{t \in V(T)}{|P_{T}(t)}|}$,
  a contradiction to the choice of~$T$.
\end{proof}

For a rooted tree decomposition~\((\calV,T,r)\) of a graph~\(G\) and a node \({t\in V(T) \setminus \{r\}}\),
the (unique) vertex in~\(V_t\setminus V_{p(t)}\) is the \emph{representative} of~\(t\).
We leave undefined the representative of~\(r\).
 
For a vertex~\(u\in V(G)\), we denote by~\(N_G(u)\) the set of neighbors of~\(u\). 
When $G$ is clear from the context, we write simply $N(u)$. 
In what follows, we denote by~\(d(u)\) the degree of~\(u\), 
by~\(N[u]\) the closed neighborhood \({N(u)\cup\{u\}}\) of \(u\), 
and by $\Delta(G)$ the maximum degree of~$G$.

\begin{remark}\label{remark:leaf-degree}
  If \(y\) is the representative of a leaf~\(t\) of a rooted tree decomposition of a graph~$G$, then~\({N_G(y)\subseteq V_t}\).
\end{remark}

\begin{proof}
  Suppose, for a contradiction, that there is a node ~$y' \in N_G(y) \setminus V_t$.
  By~(T2), there is a bag~$V_{t'}$ such that~$y,y' \in V_{t'}$,
  and hence, by~(T3), we have that $y \in V_{p(t)}$, a contradiction.
\end{proof}

\section{Graphs with treewidth at most 6}\label{section:tuza-partial-6-trees}

In this section, we verify Conjecture~\ref{conjecture:tuza} for graphs with treewidth at most~\(6\)
by extending the technique introduced by 
Puleo~\cite{Puleo15} for dealing with vertices of small degree. 
For that, we use the following definitions (see also~\cite{Puleo15}).

\begin{definition}
Given a graph \(G\), 
a nonempty set~$V_0 \subseteq V(G)$ is called \emph{reducible} if there is a set~${X \subseteq E(G)}$
and a set~$Y$ of edge-disjoint triangles in~$G$ such that the following conditions hold:
\begin{enumerate}[label = (\roman*)]
\item $|X| \leq 2|Y|$;
\item \(X \cap E(A) \neq \emptyset\) for every triangle \(A\) in $G$ that contains a vertex of~$V_0$; and
\item if \(u,v \notin V_0\) and \(uv \in E(A)\) for some \(A \in Y\), then \(uv \in X\).
\end{enumerate}
In this case we say that \((V_0,X,Y)\) is a \emph{reducing triple for \(G\)} and,
equivalently, we say that~$V_0$ is \emph{reducible using}~$X$ and~$Y$.
If there is no reducible set for~\(G\), we say that~\(G\) is \emph{irreducible}.
\end{definition}

The following lemma comes naturally (see~\cite[Lemma 2.2]{Puleo15}).

\begin{lemma}\label{lemma:tuza-partial6-reducible-puleo}
  Let \((V_0,X,Y)\) be a reducing triple for a graph \(G\) and consider ${G' = G-X-V_0}$. 
  If~${\tau(G')\leq 2\,\nu(G')}$, then~$\tau(G)\leq 2\,\nu(G)$.
\end{lemma}

A graph~\(G\) is \emph{robust} if, for every~\(v\in V(G)\),
each component of~\(G[N(v)]\) has order at least~\(5\).
The proofs of Lemma~\ref{lemma:tuza-tw6-main} and Theorem~\ref{thm:tuza-tw-6}
make frequent use of following result (see~\cite[Lemma~2.7]{Puleo15}).

\begin{lemma}\label{lemma:puleo}
  Let \(G\) be an irreducible robust graph and let~\(x,y\in V(G)\).
  The following statements hold.
  \begin{enumerate}[label = {\rm (\alph*)}]
  \item \label{puleo:a}    
    If~\(d(x) \leq 6\), then~\(\Delta\big(\overline{G[N(x)]}\big)\leq 1\) and~\(\big|E\big(\overline{G[N(x)]}\big)\big|\neq 2\);
  \item \label{puleo:b}
    If $d(x)\leq 6 $ and $d(y)\leq 6$, then $xy \notin E(G)$;
  \item \label{puleo:c}  
    If $d(x) = 7$ and~$d(y)=6$, then~${N[y]\not\subseteq N[x]}$; and
  \item \label{puleo:d}  
    If~\(d(x) \leq 8\) and~\(d(y) = 5\), then~\(N[y] \not\subseteq N[x]\).
  \end{enumerate}
\end{lemma}

We extend the result above in the following lemma.
In the pictures throughout its proof, 
given a reducing triple \((V_0,X,Y)\) for a graph \(G\), 
and two  nonadjacent vertices \(x\) and \(y\) of~\(G\), we illustrate the triangles in \(Y\) as follows.
The triangles in~\(Y\) containing \(x\) and~\(y\) are illustrated, respectively, in dashed blue and dotted green, while the triangles in~\(Y\) not containing \(x\) or \(y\) are illustrated in solid red.
The dashdotted gray lines illustrate edges that may not exist, and thin light gray lines indicate unused edges.

\begin{lemma}\label{lemma:tuza-tw6-main}
  Let~\(G\) be an irreducible robust graph and let~$x, y \in V(G)$.
  If~\({d(x) \leq 6}\), \(d(y)\leq 6\), and~\(|N(x)\cup N(y)|\leq 7\),
  then 
  \begin{enumerate}[label = {\rm (\alph*)}] 
  \item \({d(x) = d(y) = 5}\); 
  \item \({|N(x)\cap N(y)|=3}\); and 
  \item \({G[N(x)]\simeq G[N(y)]\simeq K_5}\).
  \end{enumerate}
\end{lemma}

\begin{proof}
  By Lemma~\ref{lemma:puleo}\ref{puleo:b}, we have that $xy \notin E(G)$.
  
  \medskip\noindent\textit{{\bf Proof of (a).}}
  Since \(G\) is robust, we may assume that $d(x) \geq d(y) \geq 5$ without loss of generality.
  Suppose, for a contradiction, that~$d(x) = 6$ and let~${N(x) = \{v_1,v_2,v_3,v_4,v_5,v_6\}}$.
  In what follows, we divide the proof according to the size of \(|N(x)\cap N(y)|\).
  
  \begin{description}
  \item [Case] $|N(x) \cap N(y)| = 6$. 
    
    We have \(N(y) = N(x)\).
    By Lemma~\ref{lemma:puleo}\ref{puleo:a},~$\overline{G[N(x)]} = \overline{G[N(y)]}$ is either empty, or an edge, or a matching of size~\(3\).
    Assume without loss of generality that 
    \(E\big(\overline{G[N(x)]}\big) \subseteq \{v_1v_4,v_2v_5,v_3v_6\}\).
    Let $X = E(G[N(x)])$ and~$Y = \{v_1v_3v_5,v_2v_4v_6,xv_1v_2,yv_2v_3,xv_3v_4,yv_4v_5,xv_5v_6,yv_1v_6\}$ (Figure~\ref{fig:case1}).
    Then \(|X|\leq 15 < 2|Y|\).
    Triangles in $G$ containing $x$ or \(y\) contain two vertices in $N(x)$, so contain an edge of $X$, 
    and every edge of a triangle in \(Y\) is either incident to $x$ or~$y$, or is in $X$.
    Thus~$\big(\{x,y\},X,Y\big)$ is a reducing triple for~\(G\), a contradiction.

  \item [Case] $|N(x)\cap N(y)| = 5$. 
    
    Without loss of generality, $v_1 \not\in N(y)$. 
    We may assume that \(E\big(\overline{G[N(x)]}\big) \subseteq \{v_1v_4,v_2v_5,v_3v_6\}\) by Lemma~\ref{lemma:puleo}\ref{puleo:a}.
    Moreover, if \(|N(y)| = 6\) and \(v_7\in N(y)\setminus N(x)\), then we may also assume that~\({E\big(\overline{G[N(y)]}\big)\subseteq\{v_2v_5,\,v_3v_6,\,v_4v_7,\,v_5v_7\}}\).
    Let $Z = \{xv_1,v_1v_2\}\cup E(G[\{v_2,\ldots,v_6\}])$ and~${W = \{v_2v_4v_6,\,xv_1v_2,\,yv_2v_3,\,xv_3v_4,\,yv_4v_5,\,xv_5v_6\}}$.
    For~\(|N(y)| = 5\), put \(X = Z\) and \(Y = W\), and note that \(|X|\leq 12 = 2|Y|\);
    and for \(|N(y)| = 6\), put~\(X = Z\cup \{yv_7,\,v_6v_7\}\) and \({Y = W\cup\{yv_6v_7\}}\), and note that \(|X|\leq 14 = 2|Y|\) (Figure~\ref{fig:case2}).
    Triangles in $G$ containing $x$ (resp. \(y\)) contain either~$v_1$ (resp.~\(v_7\)) or two vertices in $\{v_2,\ldots,v_6\}$, so contain an edge of $X$, and every edge of a triangle in~\(Y\) is either incident to $x$ or~$y$, or is in $X$.
    Thus~$\big(\{x,y\},X,Y\big)$ is a reducing triple for~\(G\), a contradiction.
    
  \item [Case] $|N(x)\cap N(y)| = 4$. 
    
    Without loss of generality, $v_1, v_2 \not\in N(y)$. 
    In this case, \(|N(y)| = 5\).  Let $v_7 \in N(y) \setminus N(x)$.
    By Lemma~\ref{lemma:puleo}\ref{puleo:a}, \(\overline{G[N(y)]}\) is either empty or an edge $ab$.
    If \(G[N(y)]\) is a clique, let~$ab$ be an arbitrary edge in \(G[N(y)]\).
    Let $c,d,z \in N(y) \setminus \{a,b\}$ with $c, d \in N(x)$. 
    Each one in $\{c,d,z\}$ dominates $N(y)$.
    Figures~\ref{fig:case3} and~\ref{fig:case4} depict the cases in which $v_7$ belongs or not to $\{a,b\}$, respectively.
    Let $u, w \in (N(x) \cap N(y)) \setminus \{c,d\}$.  Possibly $z \in \{u,w\}$.
    By Lemma~\ref{lemma:puleo}\ref{puleo:a}, \(\overline{G[N(x)]}\) is a matching of size at most~3, 
    so \(\overline{G[\{v_1,v_2,u,w\}]}\) is a matching of size at most~2. 
    Without loss of generality, we may assume that $v_1u, \, v_2w \in E(G)$.
    Let ${X = \{xv_1,\,xv_2,\,v_1u,\,v_2w\} \cup E(G[N(y)])}$ 
    and~${Y = \{xcd,\,yac,\,ybd,\,zad,\,zbc,\,xv_1u,\,xv_2w\}}$.
    Then~$|X| \leq 14 = 2|Y|$.
    Triangles in $G$ containing $x$ contain either $v_1$ or $v_2$, or two vertices in $N(y)$, and triangles containing $y$ contain two vertices in $N(y)$, 
    so all such triangles contain an edge of $X$.
    Also, every edge of a triangle in \(Y\) is either incident to $x$ or~$y$, or is in $X$.
    Thus~$\big(\{x,y\},X,Y\big)$ is a reducing triple for~\(G\), a contradiction.

  \begin{figure}[ht]
  \centering
  \begin{subfigure}[b]{.4\linewidth}
    \centering\scalebox{.7}{\begin{tikzpicture}[scale = 1]
{\Large
	\node (x) [black vertex] at (-4,2) {};
	\node (y) [black vertex] at (4,2) {};
	\node (label_x) [] at (-4,2.3) {$x$};
	\node (label_y) [] at (4,2.3) {$y$};

	\foreach \i in {1,...,6}{
		\node(v\i) [black vertex] at (\i*60+60:2) {};	
	}
	\foreach \i in {1,...,6}{
		\node() [] at (\i*60+60:2.3) {$v_\i$};	
	}	
	
	\draw[edge,color=red] (v2) -- (v4) -- (v6) -- (v2) (v1) -- (v3) -- (v5) -- (v1);
	\draw[edge,color=blue,densely dashed] (x) -- (v1) -- (v2) -- (x) (x) to [bend right=30] (v3) (v3) -- (v4) -- (x) (x) -- (v5) -- (v6) to [bend right=30] (x);
	\draw[edge,color=green!50!olive,line width=0.5mm,dotted] (y) -- (v2) -- (v3) -- (y) (y) -- (v5) -- (v4) to [bend right=30] (y) (y) to [bend right=30] (v1) (v1) -- (v6) -- (y);
	
	\draw[edge,color=gray,loosely dashdotted] (v1) -- (v4) (v2) -- (v5) (v3) -- (v6);
}
\end{tikzpicture}}
    \caption{}\label{fig:case1}
  \end{subfigure}
  \hspace{1cm}
  \begin{subfigure}[b]{.4\linewidth}	
    \centering\scalebox{.7}{\begin{tikzpicture}[scale = 1]
{\Large
	\node (x) [black vertex] at (-3,3) {};
	\node (y) [black vertex] at (3,3) {};
	\node (label_x) [] at (-3,3.3) {$x$};
	\node (label_y) [] at (3,3.3) {$y$};
	
	\node (v1) [black vertex] at (-4,0) {};
	\node (v7) [black vertex] at (4,0) {};
	\node () [] at (-4.4,0) {$v_1$};
	\node () [] at (4.4,0) {$v_7$};
	
	\foreach \i in {2,...,6}{
		\node(v\i) [black vertex] at (\i*360/5-18:2) {};	
	}
	\foreach \i in {3,...,5}{
		\node() [] at (\i*360/5-18:2.4) {$v_\i$};	
	}	
        \node() [] at (2*360/5-23:2.3) {$v_2$};	
        \node() [] at (6*360/5-13:2.3) {$v_6$};	
	
	\draw[edge,color=red] (v2) -- (v4) -- (v6) -- (v2);
	\draw[edge,color=blue,densely dashed] (x) -- (v1) -- (v2) -- (x) (x) -- (v3) -- (v4) -- (x) (x) to [bend right=15] (v5) (v5) -- (v6) -- (x);
	\draw[edge,color=green!50!olive,line width=0.5mm,dotted] (y) -- (v2) -- (v3) to [bend right=15] (y) (y) -- (v4) -- (v5) -- (y) (y) -- (v6) -- (v7) -- (y);
	
	\draw[edge,color=gray,loosely dashdotted] (v1) to [bend right=15] (v4) (v2) -- (v5) -- (v7) (v3) -- (v6) (v4) to [bend right=15] (v7);
}
\end{tikzpicture}}
    \caption{}\label{fig:case2}
  \end{subfigure}
  \\ 
  \begin{subfigure}[b]{.4\linewidth}	
    \centering\scalebox{.7}{\begin{tikzpicture}[scale = 1]
{\Large
	\node (x) [black vertex] at (-4,3) {};
	\node (y) [black vertex] at (3,3) {};
	\node (label_x) [] at (-4,3.3) {$x$};
	\node (label_y) [] at (3,3.3) {$y$};
	
	\foreach \i in {3,...,7}{
		\node(v\i) [black vertex] at (\i*360/5-144:2) {};	
	}
        \node() [] at (3*360/5-144:2.4) {$c$};
        \node() [] at ($(v4)+(70:.6)$)  {$u{=}z$};	
        \node() [] at (5*360/5-144:2.6) {$w{=}b$};	
        \node() [] at (6*360/5-144:2.4) {$d$};	
        \node() [] at (7*360/5-144:2.8) {$a{=}v_7$};	

	\node (v1) [black vertex] at ($(v4)+(-4,0)$) {};
	\node (v2) [black vertex] at ($(v5)+(-4,0)$) {};
	\node () [] at ($(v1)+(180:.4)$) {$v_1$};
	\node () [] at ($(v2)+(180:.4)$) {$v_2$};
	
	
	\draw[edge,color=red] (v4) -- (v3) -- (v5) -- (v4)  (v4) -- (v6) -- (v7) -- (v4);
	\draw[edge,color=blue,densely dashed] (x) -- (v3) -- (v6) -- (x) (x) -- (v1) -- (v4) -- (x) (x) -- (v2) (v2) -- (v5) -- (x);
	\draw[edge,color=green!50!olive,line width=0.5mm,dotted] (y) -- (v3) -- (v7) -- (y) (y) -- (v5) -- (v6) -- (y);
	
	\draw[edge,color=gray,loosely dashdotted] (v4) -- (v2) -- (v1) -- (v5) -- (v7);
}
\end{tikzpicture}}
    \caption{}\label{fig:case3}
  \end{subfigure}
  \hspace{1cm}
  \begin{subfigure}[b]{.4\linewidth}	
    \centering\scalebox{.7}{\begin{tikzpicture}[scale = 1]
{\Large
	\node (x) [black vertex] at (-4,3) {};
	\node (y) [black vertex] at (3,3) {};
	\node (label_x) [] at (-4,3.3) {$x$};
	\node (label_y) [] at (3,3.3) {$y$};
	
	\foreach \i in {3,...,7}{
		\node(v\i) [black vertex] at (\i*360/5-144:2) {};	
	}
        \node() [] at (3*360/5-144:2.4) {$d$};
        \node() [] at ($(v4)+(90:.6)$)  {$w{=}a$};	
        \node() [] at (5*360/5-144:2.6) {$u{=}b$};	
        \node() [] at (6*360/5-144:2.4) {$c$};	
        \node() [] at (7*360/5-144:2.7) {$z{=}v_7$};	

	\node (v1) [black vertex] at ($(v4)+(-4,0)$) {};
	\node (v2) [black vertex] at ($(v5)+(-4,0)$) {};
	\node () [] at ($(v1)+(180:.4)$) {$v_1$};
	\node () [] at ($(v2)+(180:.4)$) {$v_2$};
	
	
	\draw[edge,color=red] (v7) -- (v3) -- (v4) -- (v7)  (v7) -- (v5) -- (v6) -- (v7);
	\draw[edge,color=blue,densely dashed] (x) -- (v3) -- (v6) -- (x) (x) -- (v1) -- (v5) -- (x) (x) -- (v2) -- (v4) -- (x);
        \draw[edge,color=green!50!olive,line width=0.5mm,dotted] (y) -- (v3) -- (v5) -- (y) (y) to [bend right=20] (v4) (v4) -- (v6) -- (y);
	
	\draw[edge,color=gray,loosely dashdotted] (v1) -- (v4) -- (v5) -- (v2) -- (v1);
}
\end{tikzpicture}}
    \caption{}\label{fig:case4}
  \end{subfigure}
  \caption{Illustration of the reducing triples of the proof of Lemma~\ref{lemma:tuza-tw6-main}(a).}
  \end{figure}
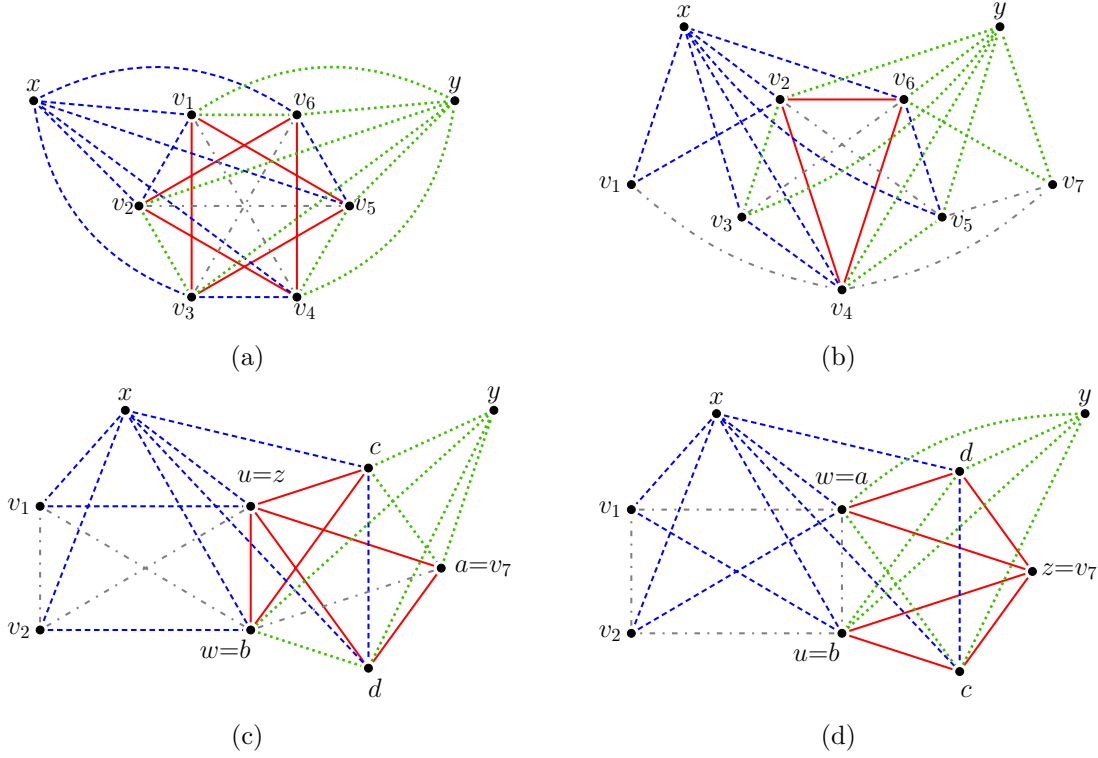
  \end{description}

\smallskip\noindent\textit{{\bf Proof of (b).}}
Clearly $|N(x) \cap N(y)| \geq 3$.  So it is enough to show that $|N(x) \cap N(y)| \leq 3$.
Let~$N(x) = \{v_1,v_2,v_3,v_4,v_5\}$ and suppose, for a contradiction, that ${N(y) = \{v_2,v_3,v_4,v_5,v_6\}}$, where possibly $v_6 = v_1$. 
By Lemma~\ref{lemma:puleo}\ref{puleo:a}, ${|E(G[\{v_2,\ldots,v_5\}])| \geq 5}$, thus we may assume, without loss of generality, that~$v_2v_3, v_3v_4, v_4v_5, v_5v_2 \in E(G)$.  
Consider
\({X = \{xv_1, \, yv_6\} \cup E(G[\{v_2,\ldots,v_5\}])}\) and 
\({Y = \{xv_2v_5, \, xv_3v_4, \, yv_2v_3, \, yv_4v_5\}}\).
Then \(|X| \leq 8 = 2|Y|\).
Triangles in \(G\) containing~\(x\) (resp.~\(y\)) either contain \(v_1\) (resp.~\(v_6\)) or contain two vertices in \(\{v_2,\ldots, v_5\}\), hence contain an edge of~\(X\).
Moreover, every edge of a triangle in \(Y\) is either incident to \(x\) or~\(y\), or is in~\(X\).    
So~$\big(\{x,y\},X,Y\big)$ is a reducing triple for~\(G\) (Figure~\ref{fig:caseb}), a contradiction.

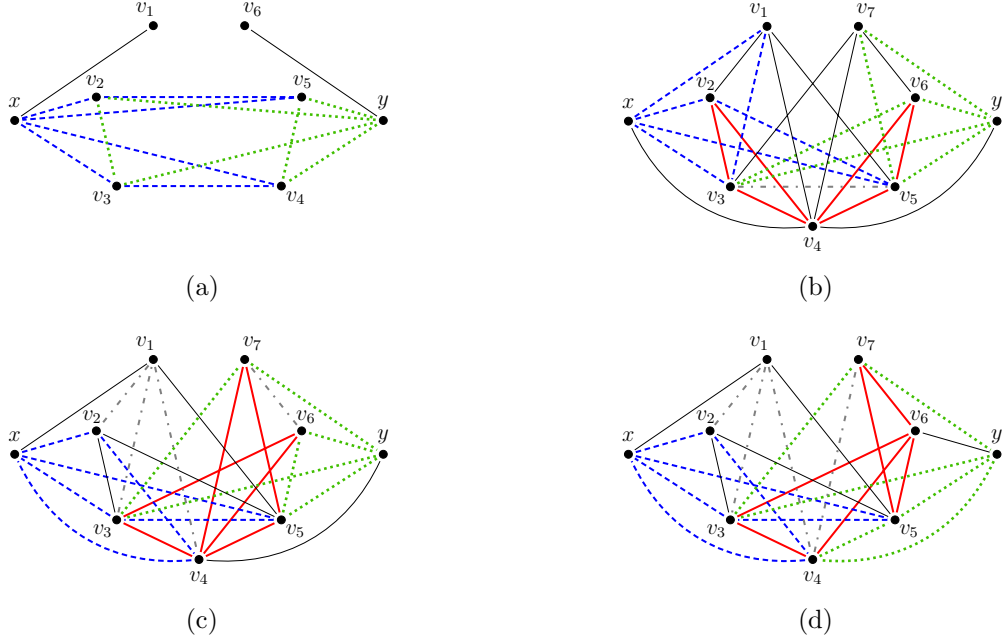
\begin{figure}[ht]
\centering
\begin{subfigure}[b]{.4\linewidth}	
  \centering\scalebox{.7}{\begin{tikzpicture}[scale = 1]

  \node (x) [black vertex] at (-3.5,0) {};
  \node (y) [black vertex] at (3.5,0) {};
  \node (label_x) [] at (-3.5,.34) {{\large $x$}};
  \node (label_y) [] at (3.5,.34) {{\large $y$}};
  
  \foreach \i in {1,...,3}{
    \node(v\i) [black vertex] at (\i*360/7+90-360/14:2) {};	
  }
  \foreach \i in {4,...,6}{
    \node(v\i) [black vertex] at (\i*360/7+90+360/14:2) {};	
  }
  \foreach \i in {1,3}{
    \node() [] at (\i*360/7+90-360/14:2.35) {{\large $v_\i$}};	
  }	
  \node() [] at (2*360/7+90-1.3*360/14:2.1) {{\large $v_2$}};	
  \node() [white] at (4*360/7+90-360/14:2.35) {{\large $v_0$}};	
  \foreach \i in {4,6}{
    \node() [] at (\i*360/7+90+360/14:2.35) {{\large $v_\i$}};	
  }	
  \node() [] at (5*360/7+90+1.3*360/14:2.1) {{\large $v_5$}};	
  
  \draw[edge,color=blue, densely dashed] (x) -- (v2) -- (v5) -- (x) (x) -- (v3) -- (v4) -- (x);
  \draw[edge,color=green!50!olive,line width=0.5mm,dotted] (y) -- (v2) -- (v3) -- (y) (y) -- (v4) -- (v5) -- (y);
  
  \draw[edge,line width=0.05mm,color=black] (v6) -- (y) (x) --(v1);
\end{tikzpicture}}
  \caption{}\label{fig:caseb}
\end{subfigure}
\hspace{13mm}
\begin{subfigure}[b]{.4\linewidth}	
  \centering\scalebox{.7}{\begin{tikzpicture}[scale = 1]

  \node (x) [black vertex] at (-3.5,0) {};
  \node (y) [black vertex] at (3.5,0) {};
  \node (label_x) [] at (-3.5,.34) {{\large $x$}};
  \node (label_y) [] at (3.5,.34) {{\large $y$}};
  
  \foreach \i in {1,...,7}{
    \node(v\i) [black vertex] at (\i*360/7+90-360/14:2) {};	
  }
  \foreach \i in {1,3,4,5,7}{
    \node() [] at (\i*360/7+90-360/14:2.34) {{\large $v_\i$}};	
  }	
  \node() [] at (2*360/7+90-1.25*360/14:2.15) {{\large $v_2$}};	
  \node() [] at (6*360/7+90-0.75*360/14:2.15) {{\large $v_6$}};	
  
  \draw[edge,color=red] (v2) -- (v3) -- (v4) -- (v2) (v4) -- (v5) -- (v6) -- (v4);
  \draw[edge,color=blue, densely dashed] (x) -- (v2) -- (v5) -- (x) (x) -- (v1) -- (v3) -- (x);
  \draw[edge,line width=0.5mm,color=green!50!olive, dotted] (y) -- (v3) -- (v6) -- (y)  (y) -- (v5) -- (v7) -- (y);
  
  \draw[edge,line width=0.05mm,color=black] (v1) -- (v2) (v1) -- (v4) (v1) -- (v5) (v3) -- (v7) (v4) -- (v7) (v6) -- (v7) (x) to [bend right=35] (v4) (y) to [bend left=35] (v4);
  \draw[edge,color=gray, loosely dashdotted] (v3) -- (v5);
\end{tikzpicture}}
  \caption{}\label{fig:casec1}
\end{subfigure}
\\ \vspace{3mm}
\begin{subfigure}[b]{.4\linewidth}	
  \centering\scalebox{.7}{\begin{tikzpicture}[scale = 1]

  \node (x) [black vertex] at (-3.5,0) {};
  \node (y) [black vertex] at (3.5,0) {};
  \node (label_x) [] at (-3.5,.34) {{\large $x$}};
  \node (label_y) [] at (3.5,.34) {{\large $y$}};
  
  \foreach \i in {1,...,7}{
    \node(v\i) [black vertex] at (\i*360/7+90-360/14:2) {};	
  }
  \foreach \i in {1,3,4,5,7}{
    \node() [] at (\i*360/7+90-360/14:2.34) {{\large $v_\i$}};	
  }	
  \node() [] at (2*360/7+90-1.25*360/14:2.15) {{\large $v_2$}};	
  \node() [] at (6*360/7+90-0.75*360/14:2.15) {{\large $v_6$}};	
  
  \draw[edge,color=red] (v3) -- (v4) -- (v6) -- (v3) (v4) -- (v5) -- (v7) -- (v4);
  \draw[edge,color=blue, densely dashed] (x) -- (v2) -- (v4) to [bend left=35] (x) (x) -- (v5) -- (v3) -- (x);
  \draw[edge,color=green!50!olive,line width=0.5mm,dotted] (y) -- (v3) -- (v7) -- (y)  (y) -- (v5) -- (v6) -- (y);
  
  \draw[edge,line width=0.05mm,color=black] (x) -- (v1) -- (v5) -- (v2) -- (v3) (y) to [bend left=35] (v4);
  \draw[edge,color=gray, loosely dashdotted] (v1) -- (v2) (v3) -- (v1) -- (v4) (v6) -- (v7);
\end{tikzpicture}}
  \caption{}\label{fig:casec2}
\end{subfigure}
\hspace{13mm}
\begin{subfigure}[b]{.4\linewidth}	
  \centering\scalebox{.7}{\begin{tikzpicture}[scale = 1]

  \node (x) [black vertex] at (-3.5,0) {};
  \node (y) [black vertex] at (3.5,0) {};
  \node (label_x) [] at (-3.5,.34) {{\large $x$}};
  \node (label_y) [] at (3.5,.34) {{\large $y$}};
  
  \foreach \i in {1,...,7}{
    \node(v\i) [black vertex] at (\i*360/7+90-360/14:2) {};	
  }
  \foreach \i in {1,3,4,5,7}{
    \node() [] at (\i*360/7+90-360/14:2.34) {{\large $v_\i$}};	
  }	
  \node() [] at (2*360/7+90-1.25*360/14:2.15) {{\large $v_2$}};	
  \node() [] at (6*360/7+90-0.75*360/14:2.15) {{\large $v_6$}};	
  
  \draw[edge,color=red] (v5) -- (v6) -- (v7) -- (v5) (v3) -- (v4) -- (v6) -- (v3);
  \draw[edge,color=blue,densely dashed] (x) -- (v2) -- (v4) to [bend left=35] (x) (x) -- (v5) -- (v3) -- (x);
  \draw[edge,color=green!50!olive,line width=0.5mm,dotted] (y) to [bend left=35] (v4) (v4) -- (v5) -- (y) (y) -- (v3) -- (v7) -- (y);
  
  \draw[edge,line width=0.05mm,color=black] (v5) -- (v2) -- (v3) (v5) -- (v1) -- (x) (v6) -- (y);
  \draw[edge,color=gray,loosely dashdotted] (v1) -- (v3) (v2) -- (v1) -- (v4) -- (v7);
\end{tikzpicture}}
  \caption{}\label{fig:casec3}
\end{subfigure}
\caption{Reducing triples of the proof of Lemma~\ref{lemma:tuza-tw6-main}(b) and~(c).}
\end{figure}

\smallskip\noindent\textit{{\bf Proof of (c).}}
Let $N(x) = \{v_1,v_2,v_3,v_4,v_5\}$ and $N(y) = \{v_3,v_4,v_5,v_6,v_7\}$.
Suppose, for a contradiction, that $G[N(y)]$ is not complete.  
By Lemma~\ref{lemma:puleo}\ref{puleo:a}, 
only two vertices in $N(y)$ are nonadjacent.
These two vertices might both be in $N(x)$, 
or only one is in $N(x)$, or none is in $N(x)$. 
So, we may assume, without loss of generality, 
that either $v_3v_5$, or $v_4v_7$, or $v_6v_7$ is not in $G$.

If $v_3v_5$ is not in $G$, then we have ${E(\overline{G[N(x)]}) = E(\overline{G[N(y)]}) = \{v_3v_5\}}$ by Lemma~\ref{lemma:puleo}\ref{puleo:a}.
In this case, consider $X = E(v_2v_3v_4)\cup E(v_4v_5v_6)\cup \{v_2v_5, \, v_1v_3, \, v_3v_6, \, v_5v_7, \, xv_1, \, yv_7\}$ and ${Y = \{v_2v_3v_4, \, v_4v_5v_6, \, xv_1v_3, \, xv_2v_5, \, yv_3v_6, \, yv_5v_7\}}$.
Every triangle in \(G\) containing \(x\) (resp.~\(y\)) either contains~\(v_1\) (resp.~\(v_7\)) or contains two vertices in \(\{v_2,\ldots, v_5\}\) (resp.~\(\{v_3,\ldots,v_6\}\)), so contains an edge of~\(X\) (Figure~\ref{fig:casec1}).

If either $v_4v_7$ or $v_6v_7$ is not in $G$, 
then we may assume, without loss of generality, 
that~$E(\overline{G[N(x)]}) \subseteq \{v_1v_2,v_1v_3,v_1v_4\}$ by Lemma~\ref{lemma:puleo}\ref{puleo:a}.  
Let $w, z \in \{v_4,v_6\}$, with ${wv_7 \in E(G)}$ and $zv_7 \not\in E(G)$, and, in this case, consider 
$X = E(G[N(y)]) \cup \{xv_1, xv_2, v_2v_4\}$ and
${Y = \{xv_3v_5, \, xv_2v_4, \, yv_3v_7, \, v_3v_4v_6, \, yv_5z, \, v_5v_7w\}}$.
Triangles in \(G\) containing \(x\) or \(y\) either contain~\(v_1\) or~\(v_2\), or contain two vertices in \(N(y)\), so contain an edge of~\(X\) (Figures~\ref{fig:casec2} and \ref{fig:casec3}).

In both cases, \(|X| = 12 = 2|Y|\) and every edge of a triangle in \(Y\) is either adjacent to~\(x\) or~\(y\), or is in~\(X\).
So~$\big(\{x,y\},X,Y\big)$ is a reducing triple for~\(G\), a contradiction.
\end{proof}

Now we are able to prove the main result of this section.

\begin{theorem}\label{thm:tuza-tw-6}
If \(G\) is a graph with treewidth at most~\(6\), then~\(\tau(G)\leq 2\,\nu(G)\).
\end{theorem}
\begin{proof}
Suppose, for a contradiction, that the statement does not hold,
and let~\(G\) be a graph with treewidth at most \(6\) and such that \(\tau(G) > 2\,\nu(G)\),
and that minimizes \(|V(G)|\) subject to these conditions.
We claim that~\(G\) is irreducible.
Indeed, suppose that \(G\) has a reducing triple \((V_0,X,Y)\),
and let~\({G'=(G-X)-V_0}\).
Note that~$G'$ has treewidth at most~6,
and, by the minimality of \(G\), we have~\({\tau(G')\leq 2\,\nu(G')}\).
By Lemma~\ref{lemma:tuza-partial6-reducible-puleo},~${\tau(G)\leq 2\,\nu(G)}$,
a contradiction.  

\begin{claimsn}\label{remark:tuz-tw6-remark-1}
  \(G\) is robust.
\end{claimsn}
\begin{proof}
  Suppose, for a contradiction, that~\(x\) is a vertex of~\(G\) such that~\(G[N(x)]\) contains a component~\(C\) with at most four vertices.
  Let~\({E_C = \{xv \colon v\in V(C)\}}\) and~\(M'\) be a maximum matching in~\(C\).

  If~\(M'=\emptyset\), then there is no triangle in~\(G\) containing the edges in~\(E_C\).
  Thus, for \(G' = G-E_C\), we have that $\tau(G') = \tau(G)$ and $\nu(G') = \nu(G)$.
  The minimality of~\(G\) implies that~\({\tau(G') \leq 2\,\nu(G')}\),
  and so~\({\tau(G)\leq 2\,\nu(G)}\), a contradiction.

  If~\({M'=\{v_1v_2\}}\), then~\(C\) is either a star or a triangle.
  If~\(C\) is a star, let~\(u\) be its center,
  and if~\(C\) is a triangle, then let~\(u\) be the vertex of~\(C\) different from~\(v_1\) and \(v_2\).
  Note that if~\(A\) is a triangle of~\(G\) containing edges of~\(E_C\),
  then~\(A\) contains~\(v_1v_2\) or~\(A\) contains~\(xu\).
  Let~\({G' = G-E(xv_1v_2)}\).
  By the minimality of~\(G\), we have~\({\tau(G') \leq 2\,\nu(G')}\).
  Thus let~\(X'\) and~\(Y'\) be a minimum triangle transversal
  and a maximum triangle packing of~\(G'\), respectively.
  Note that~\({X'\cup\{v_1v_2,xu\}}\) is a triangle transversal of~\(G\),
  and~\({Y'\cup\{xv_1v_2\}}\) is a triangle packing of~\(G\).
  Hence $$\tau(G) \leq |X'|+2 = \tau(G') + 2
  \leq 2\,\nu(G') + 2 = 2|Y'|+2 \leq 2\,\nu(G),$$
  a contradiction.

  Finally, if~\({M' = \{u_1u_2,v_1v_2\}}\),
  then we put~\({G' = G-E(xu_1u_2)-E(xv_1v_2)}\). 
  Let~\(X'\) and~\(Y'\) be a minimum triangle transversal 
  and a maximum triangle packing of~\(G'\), respectively.
  Note that~\({X'\cup \{u_1u_2,v_1v_2,xv_1,xv_2\}}\) is a triangle transversal of~\(G\),
  and~\({Y'\cup\{xu_1u_2,xv_1v_2\}}\) is a triangle packing of~\(G\).
  Hence $$\tau(G) \leq |X'|+4 = \tau(G') + 4
  \leq 2\,\nu(G') + 4 = 2|Y'|+4 \leq 2\,\nu(G),$$
  a contradiction.
\end{proof}

Suppose that~$|V(G)| \leq 7$.
As~$G$ is irreducible and robust, and \({d(v) \leq 6}\) for every~\({v\in V(G)}\),
Lemma~\ref{lemma:puleo}\ref{puleo:b} implies that $E(G) = \emptyset$.
Thus ${\tau(G)\leq 2\,\nu(G)}$, a contradiction.
So~\({|V(G)|\geq 8}\). 
By Proposition~\ref{prop:root-tree-dec},~\(G\) has a rooted tree decomposition~$(T, \mathcal{V},r)$.
Because~\(|V(G)|\geq 8\), we have that~${|V(T)|>1}$, and hence there is a node~$t \in V(T)$ with~$h(t)=1$.
If $t \neq r$, let~\(w\) be the representative of~\(t\);
otherwise let~\(w\) be an arbitrary vertex of~\(V_t\).

First suppose that~$S(t)=\{t'\}$, and let $z$ be the representative of \(t'\).
Recall that \(z \notin V_t\) and that $d(z) \leq 6$ by Remark \ref{remark:leaf-degree}.
Because \(G\) is robust, $d(z) \in \{5,6\}$.
Also, $(T-t', \mathcal{V}\setminus \{V_{t'}\},r)$ is a rooted tree decomposition for $G-z$
with $t$ as a leaf, because~$S(t)=\{t'\}$.  So we can apply Remark~\ref{remark:leaf-degree} to $G-z$ and~$t$, 
and conclude that~$N_{G-z}(w) \subseteq V_t \setminus \{w\}$.
Thus~$d(w)=|N(w)| \leq |N_{G-z}(w)| + 1 \leq |V_t| \leq 7$.
If~$d(w) = 7$,
then \(w\) must be adjacent to $z$ and to every vertex in \(V_t\setminus\{w\}\), 
so we would have that \(N[z]\subseteq N[w]\) because \(V_{t'} \subseteq V_t\cup\{z\}\).
This contradicts either Lemma~\ref{lemma:puleo}\ref{puleo:c} or Lemma~\ref{lemma:puleo}\ref{puleo:d}.
So we may assume that~$d(w) \leq 6$, 
and hence $wz \notin E(G)$ by Lemma~\ref{lemma:puleo}\ref{puleo:b}.
Thus \(N(z) \cup N(w) \subseteq V_t \setminus \{w\}\), 
and hence \(|N(z) \cup N(w)| \leq |V_t|-1 \leq 6\).
On the other hand, by Lemma~\ref{lemma:tuza-tw6-main}, 
we have that \(d(z)=d(w)=5\) and \(|N(z)\cap N(w)| = 3\), 
which imply that \(|N(z)\cup N(w)| = 7\), a contradiction.

Therefore~$|S(t)|=\ell>1$.
Let~\(S(t) = \{t_1,\ldots,t_\ell\}\) and let~\(z_i\) be the representative of~\(t_i\) for \(i=1,\ldots,\ell\). Note that $N(z_1) \cup \cdots \cup N(z_\ell) \subseteq V_t$, and $|V_t| \leq 7$. 
Thus, by Lemma~\ref{lemma:tuza-tw6-main}, \(d(z_i) = 5\) and~\(G[N(z_i)] \simeq K_5\) for every \(i \in \{1,\ldots,\ell\}\), and $|N(z_i) \cap N(z_j)| = 3$ for~\(i,j \in \{1,\ldots,\ell\}\) with \(i \neq j\).  
This implies that $|N(z_i) \cup N(z_j)| = 7$ and so $N(z_i) \cup N(z_j) = V_t$, for~\(i,j \in \{1,\ldots,\ell\}\) with \(i \neq j\). 

Suppose that $\ell = 2$.  
Note that~\(t\) is a leaf of~\((T',\calV',r)\),
where~\({T' = T-t_1-t_2}\) and~\({\calV' = \calV\setminus\{V_{t_1},V_{t_2}\}}\),
hence~\({d_{G-x-y}(w) \leq 6}\).
So \({d(w)\leq 8}\).
We may assume, without loss of generality, that~\({w \in N(z_1)}\).
Hence, because~\(G[N(z_1)]\) is a complete graph, we have that~\(N[z_1] \subseteq N[w]\), a contradiction to Lemma~\ref{lemma:puleo}\ref{puleo:d}.
We conclude that~$\ell \geq 3$.

Let~$V_t = \{v_1,\ldots,v_7\}$, with \(N(z_1) = \{v_1,\ldots,v_5\}\) and \(N(z_2) = \{v_3,\ldots,v_7\}\).
Because $N(z_3) \subseteq V_t$, \({|N(z_1) \cap N(z_3)| = |N(z_2) \cap N(z_3)| = 3}\), and~\(d(z_3) = 5\), exactly one vertex from~\(N(z_1)\cap N(z_2)\) is in~\(N(z_3)\).
So we may assume, without loss of generality, that~\(N(z_3) = \{v_1,v_2,v_3,v_6,v_7\}\).
Note that every pair of vertices in~\(V_t\) is contained in at least one~\(N(z_i)\) for~\(i\in\{1,2,3\}\).
Thus \(G[V_t] \simeq K_7\) because~\(G[N(z_i)] \simeq K_5\) for~\(i\in\{1,2,3\}\).
Let~\(X = E(G[V_t])\) and note that~\(|X|=21\).
Put
\[
\begin{array}{rll} 
  Y_1 &=&\{z_1v_1v_4,\,z_1v_2v_5,\,z_2v_3v_4,\,z_2v_5v_6,\,z_3v_1v_6,\,z_3v_2v_3\} \mbox{\ and \ }\\
  Y_2 &=&\{v_1v_2v_7,\,v_2v_4v_6,\,v_3v_6v_7,\,v_4v_5v_7,\,v_1v_3v_5\},
\end{array}
\]
and~\(Y = Y_1\cup Y_2\) (Figure~\ref{fig:main}).
Note that~\(|Y| = 11\) and~\(|X| \leq 2\,|Y|\).
Also, every triangle in~\(G\) that contains \(z_1\), \(z_2\), or \(z_3\) contains two vertices in \(\{v_1,\ldots, v_7\}\),
and hence contains an edge of~\(X\);
and every edge of a triangle in \(Y\) is either incident to \(z_1\), \(z_2\), or \(z_3\), or is in~\(X\).    
Therefore~$\big(\{z_1,z_2,z_3\},X,Y\big)$ is a reducing triple for~\(G\),  a contradiction.
This concludes the proof.
\end{proof}

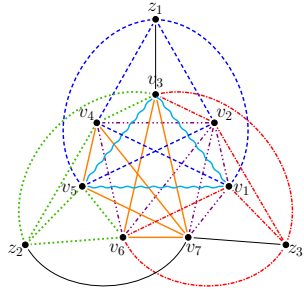
\begin{figure}[ht]
  \centering	
  \scalebox{.5}{\begin{tikzpicture}[scale = 1]

  \foreach \i in {1,...,3}{
    \node(z\i) [black vertex] at (\i*360/3+90-360/3:4) {};	
  }
  \foreach \i in {1,...,3}{
    \node() [] at (\i*360/3+90-360/3:4.3) {{\Large $z_\i$}};	
  }	
    
  \foreach \i in {1,...,7}{
    \node(v\i) [black vertex] at (\i*360/7+90-3*360/7:2) {};	
  }
  \foreach \i in {1,...,7}{
    \node() [] at (\i*360/7+90-3*360/7:2.35) {{\Large $v_\i$}};	
  }	
  
  \draw[edge,dashdotdotted,color=violet] (v1) -- (v2) -- (v7) -- (v1) (v2) -- (v4) -- (v6) -- (v2);
  \draw[edge,color=orange] (v3) -- (v6) -- (v7) -- (v3) (v4) -- (v5) -- (v7) -- (v4);
  \draw[edge,decorate,decoration={snake,amplitude=.2mm},color=cyan] (v1) -- (v3) -- (v5) -- (v1);
  \draw[edge,color=blue,densely dashed] (z1) to [bend left = 60] (v1) (v1) -- (v4) -- (z1) (z1) -- (v2) -- (v5) to [bend left = 60] (z1);
  \draw[edge,color=green!50!olive,line width=0.5mm,dotted] (z2) to [bend left = 60] (v3) (v3) -- (v4) -- (z2) (z2) -- (v5) -- (v6) -- (z2);
  \draw[edge,color=red,densely dashdotted] (z3) -- (v1) -- (v6) (v6) to [bend right = 60] (z3) (z3) -- (v2) -- (v3) to [bend left = 60] (z3);	
  \draw[edge,line width=0.05mm,color=black] (z1) -- (v3) (z2) to [bend right=60] (v7) (z3) -- (v7);

\end{tikzpicture}}
  \caption{
	Reducing triple from the proof of Theorem~\ref{thm:tuza-tw-6}.
	Triangles in \(Y\) containing \(z_1\), \(z_2\), and \(z_3\) are illustrated, 
	respectively, in dashed blue, dotted green, and dashdotted red, 
        while the remaining triangles in \(Y\) are in wavy cyan, solid orange, and dashdotdotted purple.}\label{fig:main}
\end{figure}

Theorem~\ref{thm:tuza-tw-6} strengthens the following result of Tuza~\cite{Tuza90} regarding chordal graphs.
 
\begin{proposition}{\cite[Proposition 3(b)]{Tuza90}}\label{prop:tuza90}
  If $G$ is a $K_5$-free chordal graph, then~$\tau(G) \leq 2\,\nu(G)$.
\end{proposition}

Indeed, if $G$ is chordal, then $\tw(G) \leq\omega(G) - 1$, 
where $\omega(G)$ is the size of a maximum clique in~$G$ (see~\cite[Prop.~12.3.11]{Diestel10}).
Hence, if $G$ is also $K_t$-free, then $\tw(G) \leq t - 2$, 
and Theorem~\ref{thm:tuza-tw-6} implies the following strengthening of Proposition~\ref{prop:tuza90}.

\begin{corollary}\label{corollary:tuza-chordalK8free}
  If $G$ is a $K_8$-free chordal graph, then $\tau(G) \leq 2\,\nu(G)$.
\end{corollary}

\section{Planar triangulations}\label{section:tuza-planar-triangulation}

A graph is \emph{planar} if it can be drawn in the plane so that its edges intersect only at their ends.
Such a drawing is called a \emph{planar embedding} of the graph. 
A simple graph is a \emph{maximal planar graph} if it is planar 
and adding any edge between two nonadjacent vertices destroys that property. 
In any planar embedding of a maximal planar graph, all faces are bounded by a triangle,
the so called \emph{facial triangles}, so such a graph is also referred to as a \emph{planar triangulation}. 

In this section, we prove that $\tau(G)/\nu(G) \leq 3/2$ 
for every planar triangulation~$G$ different from~$K_4$.
The proof is divided into two parts.
In Lemma~\ref{lem:tuza-planartri-tau-upper-planar} we present an upper bound for $\tau(G)$, 
and in Lemma~\ref{lem:tuza-planartri-nu-lower} we present a lower bound for~$\nu(G)$.
 
For the proof of Lemma~\ref{lem:tuza-planartri-tau-upper-planar}, 
we need the following theorem of Petersen~\cite{Petersen1891}.
A bridge is a cut edge in a graph.  

\begin{theorem}[Petersen, 1981]\label{thm:petersen}
  Every bridgeless cubic graph contains a perfect matching.
\end{theorem}

Fix a planar embedding of a maximal planar graph \(G\).
The \emph{dual graph} $G^*$ of \(G\) is the graph whose vertex set is the set of faces of $G$,
and in which two vertices are adjacent if the corresponding faces share an edge.
As $G$ is simple, $G^*$ has no bridge. Therefore~$G^*$ is a bridgeless cubic graph. 
In the next result we use that a \(2\)-connected planar graph is bipartite if and only if every face is bounded by an even cycle (see~\cite[Ch.~4, Exercise 24]{Diestel10}).

\begin{lemma}\label{lem:tuza-planartri-tau-upper-planar}
  If \(G\) is a planar graph with $n$ vertices, then \(\tau(G) \leq n-2\), 
  with equality if $G$ is a planar triangulation.
\end{lemma}
\begin{proof}
  We may assume $G$ is simple and distinct from $K_3$.  Let \(H\) be a maximal planar graph containing~\(G\).
  Note that~\(\tau(G)\leq\tau(H)\) and that $G=H$ if $G$ is a planar triangulation.
  By Euler's formula, \(H\) has~\(2n-4\) faces.
  As \(H\) is a planar triangulation, \(H^*\) is a bridgeless cubic graph on \(2n-4\) vertices.
  Thus, \(H^*\) contains a perfect matching~\(M^*\) by Theorem~\ref{thm:petersen}.
  Let \(M\) be the edges of \(H\) corresponding to the edges in~\(M^*\).
  Note that \({|M|=|M^*|=n-2}\), and that every face of~\(H\) contains precisely one edge of~\(M\).
  This implies that every face of \(H - M\) is the symmetric difference of two faces of~\(H\), and so is a cycle of length \(4\).
  Thus \(H - M\) is 2-connected and hence bipartite.
  Therefore~\(H-M\) has no triangles, and hence \(\tau(H) \leq |M| = n-2\).
  Moreover, any triangle transversal $X$ of $H$ must contain an edge 
  in each of the $2n-4$ facial triangles of $H$.  
  Each edge is in exactly two facial triangles, so $|X| \geq n-2$.  
  This implies that in fact \(\tau(H) = n-2\).
\end{proof}

In the proof of Lemma~\ref{lem:tuza-planartri-nu-lower}, 
we will use the well-known Brooks' Theorem~\cite{Brooks41}
on the \emph{chromatic number} \(\chi(G)\) of \(G\).
Recall that $\Delta(G)$ is the maximum degree of~$G$.

\begin{theorem}[Brooks, 1941]\label{thm:tuza-planartri-brooks}
  If \(G\) is a connected graph, 
  then \(\chi(G) \leq \Delta(G)\) or \(G\) is either an odd cycle or a complete graph.
\end{theorem}

An \emph{independent set} in a graph $G$ is a set of pairwise nonadjacent vertices of $G$. 
One can check that every graph $G$ has an independent set of size at least \(|V(G)|/\chi(G)\).

\begin{lemma}\label{lem:tuza-planartri-nu-lower}
  If \(G\) is a planar triangulation on $n$ vertices, different from~$K_4$, 
  then \({\nu(G)\geq \frac{2}{3}(n-2)}\).
\end{lemma}

\begin{proof}
  The dual~$G^*$ of $G$ is connected and cubic, different from~$K_4$, and has $2n-4$ vertices.
  By Theorem~\ref{thm:tuza-planartri-brooks}, $\chi(G^*) = 3$.
  Therefore \(G^*\) contains an independent set~$Y^*$ of size at least~$\frac{2n-4}{3}$.
  Let~$Y$ be the facial triangles of~\(G\) corresponding to vertices of $G^*$ in~\(Y^*\).
  Because~$Y^*$ is an independent set in~$G^*$, the set $Y$ is a triangle packing in~$G$.
  Hence $\nu(G) \geq |Y| = |Y^*| \geq \frac{2}{3}(n-2)$.
\end{proof}

In fact, the proof of Lemma~\ref{lem:tuza-planartri-nu-lower} assures that, 
in any planar triangulation different from~$K_4$, with~$f$ facial triangles, 
there exists a packing of triangles with at least $f/3$ facial triangles.  
This result will be slightly strengthened in the next subsection for a 
particular type of planar triangulation, and its strengthening will 
be used in Section~\ref{section:tuza-3-trees}, which addresses 3-trees.

The main result of this section comes directly from 
Lemmas~\ref{lem:tuza-planartri-tau-upper-planar} and~\ref{lem:tuza-planartri-nu-lower}.

\begin{theorem}\label{thm:tuza-triangulations}
  If \(G\) is a planar triangulation different from~$K_4$,
  then ${\tau(G) \leq \frac{3}{2}\,\nu(G)}$.
\end{theorem}

Note that $K_5-e$ is a planar triangulation and $\nu(K_5-e) = 2$, 
hence the bound given by Theorem~\ref{thm:tuza-triangulations} is tight, 
because $\tau(K_5-e) = 3$ by Lemma~\ref{lem:tuza-planartri-tau-upper-planar}.
In fact, there is an infinite class of planar triangulations for which this bound is tight. 
This is a consequence of Lemma~\ref{lem:tuza-planartri-tau-upper-planar} and the next lemma.

\begin{lemma}\label{lem:tight}
  Let $G$ be a planar triangulation with no separating triangle. 
  Let $H$ be the planar triangulation obtained from $G$ by adding, 
  for each facial triangle $t$ of $G$, 
  a new vertex $v_t$ adjacent to the three vertices in $t$. 
  Then $\nu(H) \leq 2n-4$, where $n$ is the number of vertices in $G$.
\end{lemma}

\begin{proof}
  Let $X$ be a triangle packing containing only facial triangles of $H$.
  Each vertex $v_t$ is in three facial triangles of $H$, 
  and each two of these three facial triangles share an edge.  
  So at most one of the three facial triangles containing $v_t$ is in $X$. 
  As there are $2n-4$ different vertices $v_t$ in~$H$, 
  the packing $X$ contains at most $2n-4$ facial triangles. 
  To complete the proof, it is enough to prove that, 
  for each triangle packing $Y$ of $H$, 
  there is a triangle packing $X$ with at least~$|Y|$ triangles, 
  consisting of only facial triangles of $H$.  
  Every triangle in $Y$ that is not facial is a triangle in $G$.  
  As~$G$ has no separating triangle, such triangle is facial in~$G$ 
  and therefore we can replace it in~$X$ by one of the facial triangles in~$H$ 
  incident to the vertex added to the corresponding face of~$G$. 
\end{proof}

A planar triangulation $H$ as in the statement of Lemma~\ref{lem:tight} has $n+2n-4 = 3n-4$ vertices, 
so Lemma~\ref{lem:tuza-planartri-nu-lower} implies that~${\nu(H) \geq \frac{2}{3}(3n-6) = 2n-4}$
and Lemma~\ref{lem:tuza-planartri-tau-upper-planar} implies that~${\tau(H) = 3n-6}$.  
Hence, $\tau(H) = \frac32\,\nu(H)$.
Moreover, there are infinitely many planar triangulations with no separating triangles 
(Figure~\ref{fig:triangulation}).

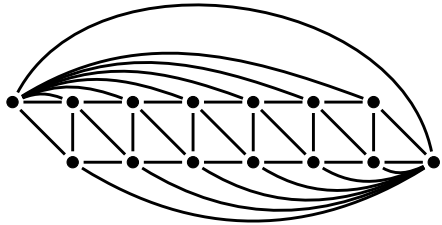
\begin{figure}[ht]
  \centering
  \begin{tikzpicture}[scale = .8]

  \foreach \i in {0,...,6}{
    \node(u\i) [black vertex] at (\i,1) {};	
  }

  \foreach \i in {0,...,6}{
    \node(v\i) [black vertex] at (\i+1,0) {};	
  }

  \foreach \i in {1,...,6}{
    \draw[edge,color=black] (u0) to [bend left=20+2*\i] (u\i);
  }

  \foreach \i in {0,...,5}{
    \draw[edge,color=black] (v6) to [bend left=32-2*\i] (v\i);
  }

  \draw[edge,color=black] (u0) to [bend left=70] (v6);

  \draw[edge,color=black] (u0) -- (u1) -- (u2) -- (u3) -- (u4) -- (u5) -- (u6);
  \draw[edge,color=black] (v0) -- (v1) -- (v2) -- (v3) -- (v4) -- (v5) -- (v6);

  \draw[edge,color=black] (u0) -- (v0) -- (u1) -- (v1) -- (u2) -- (v2) -- (u3) -- (v3) -- (u4) -- (v4) -- (u5) -- (v5) -- (u6) -- (v6);

\end{tikzpicture}
  \caption{A planar triangulation with no separating triangle.}\label{fig:triangulation}
\end{figure}

The hypothesis of $G$ having no separating triangle in Lemma~\ref{lem:tight} is necessary. 
For instance, if we consider $H$ obtained as in Lemma~\ref{lem:tight} from the planar triangulation $K_5-e$, which has a separating triangle and has $n=5$ vertices, then~$\nu(H) = 7 > 6 = 2n-4$.

\subsection{Restricted $3$-trees}

As observed after Lemma~\ref{lem:tuza-planartri-nu-lower}, 
any planar triangulation contains a packing of facial triangles with at least one third of its facial triangles.
We start by proving that, if we require the packing to contain some specific 
facial triangle, then such a packing exists with almost as many facial triangles. 

Recall that a planar triangulation can be drawn in the plane with 
any of its facial triangles as the boundary of the external face of the planar embedding.
Let us refer to the facial triangle corresponding 
to the external face of an embedding as the \emph{external facial triangle}. 

\begin{proposition}\label{prop:exconj}
  For an arbitrary planar embedding of a planar triangulation $G$ different from~$K_4$, 
  there is a triangle packing \(\calP\) of facial triangles of \(G\) containing the external facial triangle, 
  and such that \(|\calP|\geq \lceil (f-1)/3\rceil\), where~\(f\) is the number of facial triangles of $G$.
\end{proposition}

\begin{proof}
  Let \(v^*\) be the vertex of \(G^*\) corresponding to the external face of \(G\), and let \(G' = G^* - N[v^*]\).  
  By Theorem~\ref{thm:tuza-planartri-brooks}, \(G'\) contains an independent set \(I'\) of size \(\lceil(f-4)/3\rceil\),
  hence the set \(\calP\) of facial triangles of \(G\) corresponding to the vertices in \(I'\cup\{v^*\}\)
  is a triangle packing of facial triangles of~\(G\), containing the external facial triangle,
  and \(|\calP|\geq \lceil(f-4)/3\rceil+1 \geq \lceil(f-1)/3\rceil\).
\end{proof}

Figure~\ref{fig:counterexample} illustrates that Proposition~\ref{prop:exconj} is best possible. 

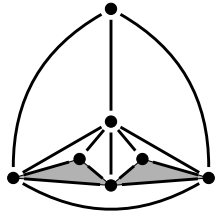
\begin{figure}[h]
  \centering
  \begin{tikzpicture}[scale = 0.5]

	\draw[fill=gray!60] (-90:1.7) -- (210:3) -- (-130:1.3) -- (-90:1.7);
	\draw[fill=gray!60] (-90:1.7) -- (-30:3) -- (-50:1.3) -- (-90:1.7);
			
	\node (a) [black vertex] at (0,0) {};
	\node (b) [black vertex] at (90:3) {};
	\node (c) [black vertex] at (210:3) {};
	\node (d) [black vertex] at (-30:3) {};
        \draw[edge] (a) -- (b) (a) -- (c) (a) -- (d) (b) to [bend right=30] (c) (c) to [bend right=30] (d) (d) to [bend right=30] (b);
        
        \node (g) [black vertex] at (-90:1.7) {};
        \draw[edge] (g) -- (a) (g) -- (c) (g) -- (d);
        \node (g1) [black vertex] at (-130:1.3) {};
       \draw[edge] (g1) -- (a) (g1) -- (c) (g1) -- (g);
        \node (g2) [black vertex] at (-50:1.3) {};
       \draw[edge] (g2) -- (a) (g2) -- (g) (g2) -- (d);
\end{tikzpicture}
  \caption{A planar triangulation whose embedding has ten faces
    and a triangle packing containing three facial triangles, one of them being the external facial triangle.}
  \label{fig:counterexample}
\end{figure}

The planar triangulation in Figure~\ref{fig:counterexample} happens to be a 3-tree. 
Next we present a special family of planar 3-trees for which we can always guarantee 
a packing containing the external facial triangle and with at least one third of its facial triangles.
This result will be used in the next section.

One can prove by induction on \(n\) that any facial triangle in a planar 3-tree \(G\) 
on \(n \geq 4\) vertices is contained in a copy of~\(K_4\) in~\(G\).
Given a planar embedding of a planar 3-tree \(G\) distinct from $K_3$,
we say that the copy of~\(K_4\) in~\(G\) that contains the external facial triangle is the \emph{root clique} of~\(G\), 
and we denote by \(r(G)\) the maximum number of vertices of \(G\) inside a face of its root clique.
For example, the graph \(G\) in Figure~\ref{fig:counterexample} is such that~\(r(G)=3\).
We say that a planar \(3\)-tree \(G\) is \emph{restricted} if it is not $K_3$ and \(r(G)=2\).
If, additionally, the root clique of \(G\) has two faces with precisely one vertex of \(G\) inside,
then we say that \(G\) is \emph{super restricted}.
Note that a restricted planar 3-tree has at least~6 and at most~10 vertices,
and hence has an even number of faces between~8 and~16.

\begin{proposition}\label{prop:restricted-3-tree}
  For an arbitrary planar embedding of a restricted (resp.\ super restricted) planar \(3\)-tree $G$, 
  there is a triangle packing \(\calP\) of facial triangles of \(G\)
  containing the external facial triangle, and such that \(|\calP|\geq \lceil f/3\rceil\) (resp.~\(|\calP|=5\)), 
  where~\(f\) is the number of faces in the embedding.
\end{proposition}

\begin{proof}
  By Proposition~\ref{prop:exconj}, there is a packing \(\calP\) of facial triangles of~\(G\) 
  including the external facial triangle, and such that \(|\calP|\geq \lceil (f-1)/3\rceil\).
  If \(f-1 \not\equiv 0 \pmod{3}\), then \(\lceil(f-1)/3\rceil = \lceil f/3\rceil\),
  and \(\calP\) is the desired triangle packing.
  So we may assume that \(f-1 \equiv 0 \pmod{3}\).
  In this case, we have \(f\in\{10,16\}\), 
  and hence~\(G\) is one of the graphs in Figures~\ref{fig:restricted-1}--\ref{fig:restricted-10}, 
  which have either ten faces and a packing with at least four facial triangles including the external facial triangle;
  or~sixteen faces and a packing of at least six facial triangles including the external facial triangle.
  If \(G\) is super restricted, 
  then~\(G\) is one of the graphs in Figures~\ref{fig:restricted-11} and~\ref{fig:restricted-12},
  which have twelve faces and a packing of at least five facial triangles including the external facial triangle.
\end{proof}

 \begin{figure}[h]
 	\centering
 	\begin{subfigure}[b]{.24\linewidth}
 		\centering
 		\begin{tikzpicture}[scale = 0.4]

	\draw[fill=gray!60] (90:3) -- (150:1.6) -- (150:.8) -- (90:3);
	\draw[fill=gray!60] (0,0) -- (150:.8) -- (210:3) -- (0,0);
	\draw[fill=gray!60] (0,0) -- (-90:1.2) -- (-30:3) -- (0,0);
		
	\node (a) [black vertex] at (0,0) {};
	\node (b) [black vertex] at (90:3) {};
	\node (c) [black vertex] at (210:3) {};
	\node (d) [black vertex] at (-30:3) {};
        \draw[edge] (a) -- (b) (a) -- (c) (a) -- (d) (b) to [bend right=30] (c) (c) to [bend right=30] (d) (d) to [bend right=30] (b);
        
        \node (e) [black vertex] at (150:.8) {};
        \draw[edge] (e) -- (a) (e) -- (b) (e) -- (c);
        \node (f) [black vertex] at (150:1.6) {};
        \draw[edge] (f) -- (e) (f) -- (b) (f) -- (c);
        \node (g) [black vertex] at (-90:1.2) {};
        \draw[edge] (g) -- (a) (g) -- (c) (g) -- (d);

\end{tikzpicture}
		\caption{}\label{fig:restricted-1}
 	\end{subfigure}
 	\begin{subfigure}[b]{.24\linewidth}
 		\centering
 		\begin{tikzpicture}[scale = 0.4]

	\draw[fill=gray!60] (90:3) -- (150:1.6) -- (0,0) -- (90:3);
	\draw[fill=gray!60] (0,0) -- (180:1.2) -- (210:3) -- (0,0);
	\draw[fill=gray!60] (0,0) -- (-90:1.2) -- (-30:3) -- (0,0);
		
	\node (a) [black vertex] at (0,0) {};
	\node (b) [black vertex] at (90:3) {};
	\node (c) [black vertex] at (210:3) {};
	\node (d) [black vertex] at (-30:3) {};
        \draw[edge] (a) -- (b) (a) -- (c) (a) -- (d) (b) to [bend right=30] (c) (c) to [bend right=30] (d) (d) to [bend right=30] (b);
        
        \node (e) [black vertex] at (150:1.6) {};
        \draw[edge] (e) -- (a) (e) -- (b) (e) -- (c);
        \node (f) [black vertex] at (180:1.2) {};
        \draw[edge] (f) -- (e) (f) -- (a) (f) -- (c);
        \node (g) [black vertex] at (-90:1.2) {};
        \draw[edge] (g) -- (a) (g) -- (c) (g) -- (d);

\end{tikzpicture}
		\caption{}\label{fig:restricted-2}
 	\end{subfigure}
 	\begin{subfigure}[b]{.24\linewidth}
 		\centering
 		\begin{tikzpicture}[scale = 0.4]

	\draw[fill=gray!60] (90:3) -- (150:1.6) -- (120:1.2) -- (90:3);
	\draw[fill=gray!60] (0,0) -- (150:1.6) -- (210:3) -- (0,0);
	\draw[fill=gray!60] (0,0) -- (-90:1.2) -- (-30:3) -- (0,0);
		
	\node (a) [black vertex] at (0,0) {};
	\node (b) [black vertex] at (90:3) {};
	\node (c) [black vertex] at (210:3) {};
	\node (d) [black vertex] at (-30:3) {};
        \draw[edge] (a) -- (b) (a) -- (c) (a) -- (d) (b) to [bend right=30] (c) (c) to [bend right=30] (d) (d) to [bend right=30] (b);
        
        \node (e) [black vertex] at (150:1.6) {};
        \draw[edge] (e) -- (a) (e) -- (b) (e) -- (c);
        \node (f) [black vertex] at (120:1.2) {};
        \draw[edge] (f) -- (e) (f) -- (a) (f) -- (b);
        \node (g) [black vertex] at (-90:1.2) {};
        \draw[edge] (g) -- (a) (g) -- (c) (g) -- (d);

\end{tikzpicture}
		\caption{}\label{fig:restricted-3}
 	\end{subfigure}
 	\begin{subfigure}[b]{.24\linewidth}
 		\centering
 		\begin{tikzpicture}[scale = 0.4]

	\draw[fill=gray!60] (90:3) -- (150:1.6) -- (150:.8) -- (90:3);
	\draw[fill=gray!60] (0,0) -- (150:.8) -- (210:3) -- (0,0);
	\draw[fill=gray!60] (210:3) -- (-90:1.6) -- (-90:.8) -- (210:3);
	\draw[fill=gray!60] (0,0) -- (-90:.8) -- (-30:3) -- (0,0);
	\draw[fill=gray!60] (-30:3) -- (30:1.6) -- (30:.8) -- (-30:3);
	\draw[fill=gray!60] (0,0) -- (30:.8) -- (90:3) -- (0,0);	
		
	\node (a) [black vertex] at (0,0) {};
	\node (b) [black vertex] at (90:3) {};
	\node (c) [black vertex] at (210:3) {};
	\node (d) [black vertex] at (-30:3) {};
        \draw[edge] (a) -- (b) (a) -- (c) (a) -- (d) (b) to [bend right=30] (c) (c) to [bend right=30] (d) (d) to [bend right=30] (b);
        
        \node (e1) [black vertex] at (150:.8) {};
        \draw[edge] (e1) -- (a) (e1) -- (b) (e1) -- (c);
        \node (f1) [black vertex] at (150:1.6) {};
        \draw[edge] (f1) -- (e1) (f1) -- (b) (f1) -- (c);

        \node (e2) [black vertex] at (-90:.8) {};
        \draw[edge] (e2) -- (a) (e2) -- (c) (e2) -- (d);
        \node (f2) [black vertex] at (-90:1.6) {};
        \draw[edge] (f2) -- (e2) (f2) -- (c) (f2) -- (d);

        \node (e3) [black vertex] at (30:.8) {};
        \draw[edge] (e3) -- (a) (e3) -- (d) (e3) -- (b);
        \node (f3) [black vertex] at (30:1.6) {};
        \draw[edge] (f3) -- (e3) (f3) -- (d) (f3) -- (b);

\end{tikzpicture}
		\caption{}\label{fig:restricted-4}
 	\end{subfigure}	
 	
 	\bigskip	
 	\begin{subfigure}[b]{.24\linewidth}
 		\centering
 		\begin{tikzpicture}[scale = 0.4]

	\draw[fill=gray!60] (0,0) -- (180:1.2) -- (210:3) -- (0,0);
	\draw[fill=gray!60] (210:3) -- (-90:1.6) -- (-90:.8) -- (210:3);
	\draw[fill=gray!60] (0,0) -- (-90:.8) -- (-30:3) -- (0,0);
	\draw[fill=gray!60] (-30:3) -- (30:1.6) -- (30:.8) -- (-30:3);
	\draw[fill=gray!60] (0,0) -- (30:.8) -- (90:3) -- (0,0);	
		
	\node (a) [black vertex] at (0,0) {};
	\node (b) [black vertex] at (90:3) {};
	\node (c) [black vertex] at (210:3) {};
	\node (d) [black vertex] at (-30:3) {};
        \draw[edge] (a) -- (b) (a) -- (c) (a) -- (d) (b) to [bend right=30] (c) (c) to [bend right=30] (d) (d) to [bend right=30] (b);
        
        \node (e1) [black vertex] at (150:1.6) {};
        \draw[edge] (e1) -- (a) (e1) -- (b) (e1) -- (c);
        \node (f1) [black vertex] at (180:1.2) {};
	\draw[edge] (f1) -- (e1) (f1) -- (a) (f1) -- (c);

        \node (e2) [black vertex] at (-90:.8) {};
        \draw[edge] (e2) -- (a) (e2) -- (c) (e2) -- (d);
        \node (f2) [black vertex] at (-90:1.6) {};
        \draw[edge] (f2) -- (e2) (f2) -- (c) (f2) -- (d);

        \node (e3) [black vertex] at (30:.8) {};
        \draw[edge] (e3) -- (a) (e3) -- (d) (e3) -- (b);
        \node (f3) [black vertex] at (30:1.6) {};
        \draw[edge] (f3) -- (e3) (f3) -- (d) (f3) -- (b);

\end{tikzpicture}
		\caption{}\label{fig:restricted-5}
 	\end{subfigure}
 	\begin{subfigure}[b]{.24\linewidth}
 		\centering
 		\begin{tikzpicture}[scale = 0.4]

	\draw[fill=gray!60] (150:1.6) -- (180:1.2) -- (0,0) -- (150:1.6);
	\draw[fill=gray!60] (210:3) -- (0,0) -- (-120:1.2) -- (210:3);
	\draw[fill=gray!60] (0,0) -- (-90:1.6) -- (-30:3) -- (0,0);
	\draw[fill=gray!60] (-30:3) -- (30:1.6) -- (30:.8) -- (-30:3);
	\draw[fill=gray!60] (0,0) -- (30:.8) -- (90:3) -- (0,0);	
		
	\node (a) [black vertex] at (0,0) {};
	\node (b) [black vertex] at (90:3) {};
	\node (c) [black vertex] at (210:3) {};
	\node (d) [black vertex] at (-30:3) {};
        \draw[edge] (a) -- (b) (a) -- (c) (a) -- (d) (b) to [bend right=30] (c) (c) to [bend right=30] (d) (d) to [bend right=30] (b);
        
        \node (e1) [black vertex] at (150:1.6) {};
        \draw[edge] (e1) -- (a) (e1) -- (b) (e1) -- (c);
        \node (f1) [black vertex] at (180:1.2) {};
	\draw[edge] (f1) -- (e1) (f1) -- (a) (f1) -- (c);

        \node (e2) [black vertex] at (-90:1.6) {};
        \draw[edge] (e2) -- (a) (e2) -- (c) (e2) -- (d);
        \node (f2) [black vertex] at (-120:1.2) {};
        \draw[edge] (f2) -- (e2) (f2) -- (c) (f2) -- (a);

        \node (e3) [black vertex] at (30:.8) {};
        \draw[edge] (e3) -- (a) (e3) -- (d) (e3) -- (b);
        \node (f3) [black vertex] at (30:1.6) {};
        \draw[edge] (f3) -- (e3) (f3) -- (d) (f3) -- (b);

\end{tikzpicture}
		\caption{}\label{fig:restricted-6}
 	\end{subfigure}
 	\begin{subfigure}[b]{.24\linewidth}
 		\centering
 		\begin{tikzpicture}[scale = 0.4]

	\draw[fill=gray!60] (150:1.6) -- (180:1.2) -- (0,0) -- (150:1.6);
	\draw[fill=gray!60] (210:3) -- (0,0) -- (-90:1.6) -- (210:3);
	\draw[fill=gray!60] (0,0) -- (-60:1.2) -- (-30:3) -- (0,0);
	\draw[fill=gray!60] (-30:3) -- (30:1.6) -- (30:.8) -- (-30:3);
	\draw[fill=gray!60] (0,0) -- (30:.8) -- (90:3) -- (0,0);	
		
	\node (a) [black vertex] at (0,0) {};
	\node (b) [black vertex] at (90:3) {};
	\node (c) [black vertex] at (210:3) {};
	\node (d) [black vertex] at (-30:3) {};
        \draw[edge] (a) -- (b) (a) -- (c) (a) -- (d) (b) to [bend right=30] (c) (c) to [bend right=30] (d) (d) to [bend right=30] (b);
        
        \node (e1) [black vertex] at (150:1.6) {};
        \draw[edge] (e1) -- (a) (e1) -- (b) (e1) -- (c);
        \node (f1) [black vertex] at (180:1.2) {};
	\draw[edge] (f1) -- (e1) (f1) -- (a) (f1) -- (c);

        \node (e2) [black vertex] at (-90:1.6) {};
        \draw[edge] (e2) -- (a) (e2) -- (c) (e2) -- (d);
        \node (f2) [black vertex] at (-60:1.2) {};
        \draw[edge] (f2) -- (e2) (f2) -- (d) (f2) -- (a);

        \node (e3) [black vertex] at (30:.8) {};
        \draw[edge] (e3) -- (a) (e3) -- (d) (e3) -- (b);
        \node (f3) [black vertex] at (30:1.6) {};
        \draw[edge] (f3) -- (e3) (f3) -- (d) (f3) -- (b);

\end{tikzpicture}
		\caption{}\label{fig:restricted-7}
 	\end{subfigure}
 	\begin{subfigure}[b]{.24\linewidth}
 		\centering
 		\begin{tikzpicture}[scale = 0.4]

	\draw[fill=gray!60] (150:1.6) -- (120:1.2) -- (0,0) -- (150:1.6);
	\draw[fill=gray!60] (210:3) -- (0,0) -- (-90:1.6) -- (210:3);
	\draw[fill=gray!60] (0,0) -- (-60:1.2) -- (-30:3) -- (0,0);
	\draw[fill=gray!60] (-30:3) -- (30:1.6) -- (30:.8) -- (-30:3);
	\draw[fill=gray!60] (0,0) -- (30:.8) -- (90:3) -- (0,0);	
		
	\node (a) [black vertex] at (0,0) {};
	\node (b) [black vertex] at (90:3) {};
	\node (c) [black vertex] at (210:3) {};
	\node (d) [black vertex] at (-30:3) {};
        \draw[edge] (a) -- (b) (a) -- (c) (a) -- (d) (b) to [bend right=30] (c) (c) to [bend right=30] (d) (d) to [bend right=30] (b);
        
        \node (e1) [black vertex] at (150:1.6) {};
        \draw[edge] (e1) -- (a) (e1) -- (b) (e1) -- (c);
        \node (f1) [black vertex] at (120:1.2) {};
	\draw[edge] (f1) -- (e1) (f1) -- (a) (f1) -- (b);

        \node (e2) [black vertex] at (-90:1.6) {};
        \draw[edge] (e2) -- (a) (e2) -- (c) (e2) -- (d);
        \node (f2) [black vertex] at (-60:1.2) {};
        \draw[edge] (f2) -- (e2) (f2) -- (d) (f2) -- (a);

        \node (e3) [black vertex] at (30:.8) {};
        \draw[edge] (e3) -- (a) (e3) -- (d) (e3) -- (b);
        \node (f3) [black vertex] at (30:1.6) {};
        \draw[edge] (f3) -- (e3) (f3) -- (d) (f3) -- (b);

\end{tikzpicture}
		\caption{}\label{fig:restricted-8}
 	\end{subfigure}
 	
 	\bigskip	
 	\begin{subfigure}[b]{.24\linewidth}
 		\centering
 		\begin{tikzpicture}[scale = 0.4]

	\draw[fill=gray!60] (150:1.6) -- (120:1.2) -- (0,0) -- (150:1.6);
	\draw[fill=gray!60] (210:3) -- (0,0) -- (-90:1.6) -- (210:3);
	\draw[fill=gray!60] (0,0) -- (-60:1.2) -- (-30:3) -- (0,0);
	\draw[fill=gray!60] (-30:3) -- (30:1.6) -- (0:1.2) -- (-30:3);
	\draw[fill=gray!60] (0,0) -- (30:1.6) -- (90:3) -- (0,0);	
		
	\node (a) [black vertex] at (0,0) {};
	\node (b) [black vertex] at (90:3) {};
	\node (c) [black vertex] at (210:3) {};
	\node (d) [black vertex] at (-30:3) {};
        \draw[edge] (a) -- (b) (a) -- (c) (a) -- (d) (b) to [bend right=30] (c) (c) to [bend right=30] (d) (d) to [bend right=30] (b);
        
        \node (e1) [black vertex] at (150:1.6) {};
        \draw[edge] (e1) -- (a) (e1) -- (b) (e1) -- (c);
        \node (f1) [black vertex] at (120:1.2) {};
	\draw[edge] (f1) -- (e1) (f1) -- (a) (f1) -- (b);

        \node (e2) [black vertex] at (-90:1.6) {};
        \draw[edge] (e2) -- (a) (e2) -- (c) (e2) -- (d);
        \node (f2) [black vertex] at (-60:1.2) {};
        \draw[edge] (f2) -- (e2) (f2) -- (d) (f2) -- (a);

        \node (e3) [black vertex] at (30:1.6) {};
        \draw[edge] (e3) -- (a) (e3) -- (d) (e3) -- (b);
        \node (f3) [black vertex] at (0:1.2) {};
        \draw[edge] (f3) -- (e3) (f3) -- (d) (f3) -- (a);

\end{tikzpicture}
		\caption{}\label{fig:restricted-9}
 	\end{subfigure}
 	\begin{subfigure}[b]{.24\linewidth}
 		\centering
 		\begin{tikzpicture}[scale = 0.4]

	\draw[fill=gray!60] (90:3) -- (150:1.6) -- (0,0) -- (90:3);
	\draw[fill=gray!60] (150:1.6) -- (180:1.2) -- (210:3) -- (150:1.6);
	\draw[fill=gray!60] (210:3) -- (0,0) -- (-90:1.6) -- (210:3);
	\draw[fill=gray!60] (-90:1.6) -- (-60:1.2) -- (-30:3) -- (-90:1.6);
	\draw[fill=gray!60] (-30:3) -- (30:1.6) -- (0,0) -- (-30:3);
	\draw[fill=gray!60] (30:1.6) -- (60:1.2) -- (90:3) -- (30:1.6);	
		
	\node (a) [black vertex] at (0,0) {};
	\node (b) [black vertex] at (90:3) {};
	\node (c) [black vertex] at (210:3) {};
	\node (d) [black vertex] at (-30:3) {};
        \draw[edge] (a) -- (b) (a) -- (c) (a) -- (d) (b) to [bend right=30] (c) (c) to [bend right=30] (d) (d) to [bend right=30] (b);
        
        \node (e1) [black vertex] at (150:1.6) {};
        \draw[edge] (e1) -- (a) (e1) -- (b) (e1) -- (c);
        \node (f1) [black vertex] at (180:1.2) {};
	\draw[edge] (f1) -- (e1) (f1) -- (a) (f1) -- (c);

        \node (e2) [black vertex] at (-90:1.6) {};
        \draw[edge] (e2) -- (a) (e2) -- (c) (e2) -- (d);
        \node (f2) [black vertex] at (-60:1.2) {};
        \draw[edge] (f2) -- (e2) (f2) -- (d) (f2) -- (a);

        \node (e3) [black vertex] at (30:1.6) {};
        \draw[edge] (e3) -- (a) (e3) -- (d) (e3) -- (b);
        \node (f3) [black vertex] at (60:1.2) {};
        \draw[edge] (f3) -- (e3) (f3) -- (a) (f3) -- (b);

\end{tikzpicture}
		\caption{}\label{fig:restricted-10}
 	\end{subfigure}
 	\begin{subfigure}[b]{.24\linewidth}
 		\centering
 		\begin{tikzpicture}[scale = 0.4]

	\draw[fill=gray!60] (90:3) -- (0,0) -- (150:.8) -- (90:3);
	\draw[fill=gray!60] (150:1.6) -- (150:.8) -- (210:3) -- (150:1.6);
	\draw[fill=gray!60] (0,0) -- (-90:1.2) -- (210:3) -- (0,0);
	\draw[fill=gray!60] (0,0) -- (30:1.2) -- (-30:3) -- (0,0);
		
	\node (a) [black vertex] at (0,0) {};
	\node (b) [black vertex] at (90:3) {};
	\node (c) [black vertex] at (210:3) {};
	\node (d) [black vertex] at (-30:3) {};
        \draw[edge] (a) -- (b) (a) -- (c) (a) -- (d) (b) to [bend right=30] (c) (c) to [bend right=30] (d) (d) to [bend right=30] (b);
        
        \node (e) [black vertex] at (150:.8) {};
        \draw[edge] (e) -- (a) (e) -- (b) (e) -- (c);
        \node (f) [black vertex] at (150:1.6) {};
        \draw[edge] (f) -- (e) (f) -- (b) (f) -- (c);
        \node (g) [black vertex] at (-90:1.2) {};
        \draw[edge] (g) -- (a) (g) -- (c) (g) -- (d);
        \node (h) [black vertex] at (30:1.2) {};
        \draw[edge] (h) -- (a) (h) -- (d) (h) -- (b);
\end{tikzpicture}
		\caption{}\label{fig:restricted-11}
 	\end{subfigure}
 	\begin{subfigure}[b]{.24\linewidth}
 		\centering
 		\begin{tikzpicture}[scale = 0.4]

	\draw[fill=gray!60] (90:3) -- (150:1.6) -- (0,0) -- (90:3);
	\draw[fill=gray!60] (150:1.6) -- (180:1.2) -- (210:3) -- (150:1.6);
	\draw[fill=gray!60] (0,0) -- (-90:1.2) -- (210:3) -- (0,0);
	\draw[fill=gray!60] (0,0) -- (30:1.2) -- (-30:3) -- (0,0);
			
	\node (a) [black vertex] at (0,0) {};
	\node (b) [black vertex] at (90:3) {};
	\node (c) [black vertex] at (210:3) {};
	\node (d) [black vertex] at (-30:3) {};
        \draw[edge] (a) -- (b) (a) -- (c) (a) -- (d) (b) to [bend right=30] (c) (c) to [bend right=30] (d) (d) to [bend right=30] (b);
        
        \node (e) [black vertex] at (150:1.6) {};
        \draw[edge] (e) -- (a) (e) -- (b) (e) -- (c);
        \node (f) [black vertex] at (180:1.2) {};
        \draw[edge] (f) -- (e) (f) -- (a) (f) -- (c);
        \node (g) [black vertex] at (-90:1.2) {};
        \draw[edge] (g) -- (a) (g) -- (c) (g) -- (d);
        \node (h) [black vertex] at (30:1.2) {};
        \draw[edge] (h) -- (a) (h) -- (d) (h) -- (b);
\end{tikzpicture}
		\caption{}\label{fig:restricted-12}
 	\end{subfigure} 	
 	\caption{Restricted and super restricted planar $3$-trees of the 
 		proof of Proposition~\ref{prop:restricted-3-tree}.}\label{fig:restricted-3-tree}	
\end{figure}
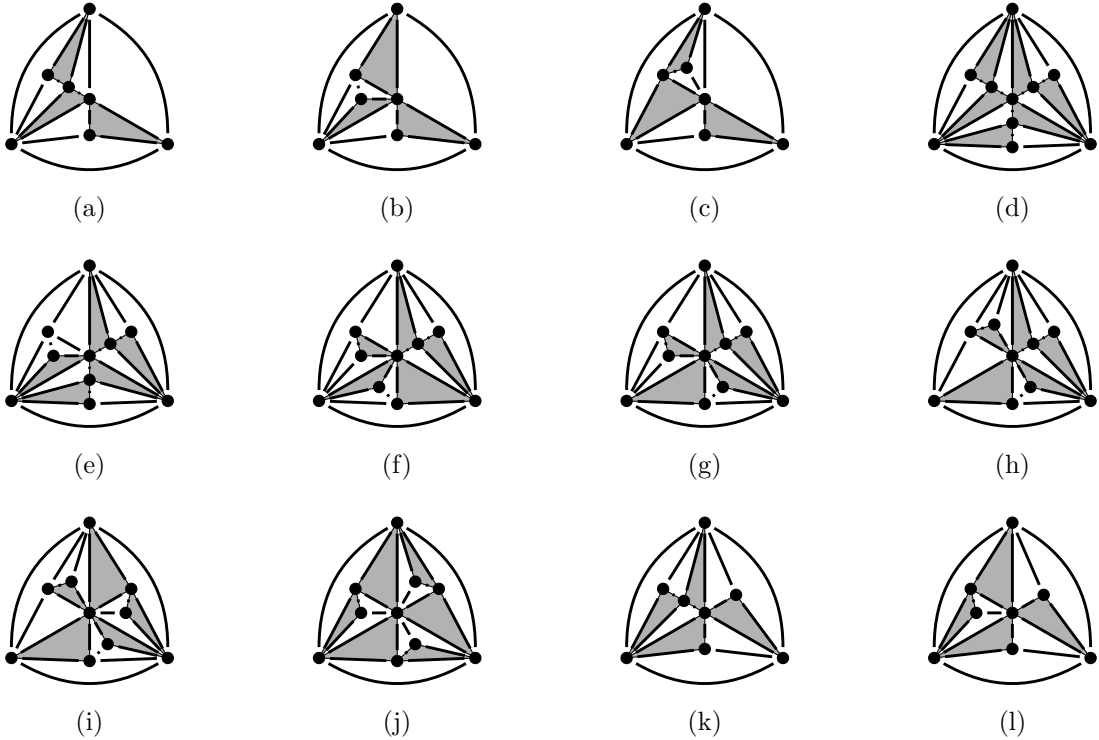

\section{3-Trees}\label{section:tuza-3-trees}

In this section, we prove that~$\tau(G)\leq \frac{9}{5}\,\nu(G) + \frac{1}{5}$ for every \(3\)-tree \(G\). 
For a graph \(G\), we say that the pair~$(X,Y)$ is a~$\frac{9}{5}$-\emph{TP} of~$G$ 
if~$X$ is a triangle transversal, $Y$ is a triangle packing of \(G\), 
and~${|X| \leq \frac{9}{5} |Y| + \frac{1}{5}}$.
If~$G$ has a~$\frac{9}{5}$-TP, then~$\tau(G)\leq \frac{9}{5}\,\nu(G)+\frac{1}{5}$.

Let~${(T, \mathcal{V},r)}$ be a rooted tree decomposition of a graph~$G$.
For a node~${t \in V(T) \setminus \{r\}}$, 
we denote by~$R(t)$ the set of all representatives of the descendants of~$t$ in $T$.
Note that the representative of~\(t\) is also in \(R(t)\).
Recall that~$S(t)$ is the set of successors of~$t$.
For every triple of vertices~${\Delta \subseteq V_t}$,
let~${S^{\Delta}(t)=\{t' \in S(t)\colon V_{t'} \cap V_t = \Delta\}}$.
When $t$ is clear from the context, we simply write~$S^{\Delta}$.

Our proof relies on the analysis of nodes of $T$ with small height,
and guarantees that a minimal counterexample has a particular configuration.

\begin{theorem}\label{thm:tuza-3-trees}
  If~$G$ is a 3-tree, then~$\tau(G)\leq \frac{9}{5}\,\nu(G) + \frac{1}{5}$.
\end{theorem}

\begin{proof}
  The statement holds if $|V(G)| \leq 6$. 
  Indeed, if~$|V(G)|=4$, then $\tau(G)=2$ and~$\nu(G) = 1$; 
  if~$|V(G)|=5$, then $\tau(G)=3$ and~$\nu(G) = 2$;
  and if~$|V(G)|=6$, then $\tau(G) \leq 4$ and~$\nu(G) = 3$.

  Suppose, for a contradiction, that the statement does not hold and 
  let~$G$ be a minimal counterexample.  Then $|V(G)| \geq 7$ and thus 
  every full tree decomposition of~$G$ of width 3 has at least~4 nodes. 
  Let~$(T, \mathcal{V},r)$ be a rooted tree decomposition of~$G$ of width~3,
  with $r$ being a node of degree~1 in $T$.  
  Because $|V(T)| \geq 4$ and $|S(r)| = 1$, we have $h(r) > 1$.
  
  In what follows, we present several claims regarding~\(G\) and~$(T, \mathcal{V},r)$.

  \begin{claim} \label{claim:tuza-minimal-h(t)=1-then-S(t)=1}
    Every node $t$ of $T$ with $h(t) = 1$ is such that $|S(t)|=1$.
  \end{claim}

  \begin{proof}
    Let~$S(t)=\{t_1,\ldots,t_k\}$ and, for a contradiction, suppose that $k \geq 2$.
    Let~$v_i$ be the representative of~$t_i$ for $i = 1,\ldots,k$.
    Let~$V_t=\{a,b,c,d\}$ with~$V_t \cap V_{p(t)}=\{b,c,d\}$.
    Since~$h(t) = 1$, 
    at least one triangle $\Delta$ in~$\{abc,abd,acd\}$ is such that~${S^\Delta \neq \emptyset}$.
    Let~${G'=G-R(t)}$ and note that~$G'$ is a~3-tree.
    By the minimality of~$G$, there exists a~$\frac{9}{5}$-TP $(X',Y')$ of~$G'$.
    If exactly one triangle~$\Delta$ in~$\{abc,abd,acd\}$, say~$\Delta=abc$, 
    is such that~$S^\Delta \neq \emptyset$ (Figure~\ref{fig:tuza-3trees-lemma-a}), then 
    ${\big(X' \cup \{ab,bc,ac\},Y' \cup  \{acv_1,abv_2\}\big)}$ is a~$\frac{9}{5}$-TP of~$G$, 
    because $\tau(G) \leq \tau(G') + 3 \leq \frac{9}{5}\,\nu(G') + \frac{16}{5}
    \leq \frac{9}{5}(\nu(G)-2) + \frac{16}{5} < \frac{9}{5}\,\nu(G) + \frac{1}{5}$.
    This is a contradiction, 
    so at most one triangle~$\Delta$ in~$\{abc,abd,acd\}$ is such that~$S^\Delta = \emptyset$.

    Assume, without loss of generality, that~${t_1 \in S^{abc}}$ and~${t_2 \in S^{abd}}$.
    Suppose that ${|S(t)|=2}$ and let~${e \in X' \cap E(bcd)}$.
    Without loss of generality, either \(e=bc\) or~\(e=cd\). 
    If~${e = bc}$, then let~${X=X' \cup \{ad,av_1,bv_2\}}$.
    If~${e = cd}$, then let~${X=X' \cup \{ab,cv_1,dv_2\}}$ (Figure~\ref{fig:tuza-3trees-lemma-b}).
    In both cases,~${\big(X, Y' \cup \{acv_1,abv_2\}\big)}$ is a~$\frac{9}{5}$-TP of~$G$, 
    a contradiction.
    Therefore $|S(t)| \geq 3$.    

    Assume, without loss of generality, that either~${t_3 \in S^{abd}}$ or~${t_3 \in S^{acd}}$.
    Note that ${E(bcd) \cap X' \neq \emptyset}$.
    Then~$\big(X' \cup \{ab,bc,cd,ac,bd,ad\}, Y' \cup \{acv_1,abv_2\,adv_3\}\big)$
    is a~$\frac{9}{5}$-TP of~$G$, a contradiction (Figure~\ref{fig:tuza-3trees-lemma-c}).
  \end{proof}
  
\begin{figure}[h]
  \centering
  \begin{subfigure}[b]{.3\linewidth}
    \centering
    \scalebox{.7}{\begin{tikzpicture}[scale = 0.6]

	\draw[draw=none, fill=gray!60] (0,0) -- (90:3) -- (120:1.5) -- (0,0);
	\draw[draw=none, fill=gray!60] (0,0) -- (180:1.5) -- (210:3) -- (0,0);
		
	\node (a) [black vertex] at (0,0) {};
	\node (b) [black vertex] at (90:3) {};
	\node (c) [black vertex] at (210:3) {};
	\node (d) [black vertex] at (-30:3) {};

	\node () [] at (.3,.3) {$a$};
	\node () [] at (90:3.6) {$b$};
	\node () [] at (210:3.4) {$c$};
	\node () [] at (-30:3.3) {$d$};

        \draw[edge] (a)  (b) (a)  (c) (a) -- (d) (b)  (c) (c) to [bend right=30] (d) (d) to [bend right=30] (b);
	\draw[edge,dashed] (a) -- (b) (b) to [bend right=30] (c) (c) -- (a);
        
        \node (e) [black vertex] at (120:1.5) {};
        \draw[edge] (e) -- (a) (e) -- (b) (e) to [bend right=20] (c);
        \node (f) [black vertex] at (180:1.5) {};
        \draw[edge] (f) -- (a) (f) to [bend left=20] (b) (f) -- (c);

\end{tikzpicture}}
    \caption{}\label{fig:tuza-3trees-lemma-a}
  \end{subfigure}
  \begin{subfigure}[b]{.3\linewidth}
    \centering
    \scalebox{.7}{\begin{tikzpicture}[scale = 0.6]

	\draw[draw=none, fill=gray!60] (0,0) -- (90:3) -- (30:1.6) -- (0,0);
	\draw[draw=none, fill=gray!60] (0,0) -- (150:1.6) -- (210:3) -- (0,0);
		
	\node (a) [black vertex] at (0,0) {};
	\node (b) [black vertex] at (90:3) {};
	\node (c) [black vertex] at (210:3) {};
	\node (d) [black vertex] at (-30:3) {};

	\node () [] at (0,-.4) {$a$};
	\node () [] at (90:3.6) {$b$};
	\node () [] at (210:3.4) {$c$};
	\node () [] at (-30:3.3) {$d$};

        \draw[edge]   (c) -- (a)   (d) to [bend right=30] (b);
        
        \node (e) [black vertex] at (150:1.6) {};
        \draw[edge]  (e) -- (b);
        \node (f) [black vertex] at (30:1.6) {};
        \draw[edge] (f) -- (a);

	\draw[edge,dashed] (c) to [bend right=30] (d) (a) -- (b)  (e) -- (c)  (f) -- (d);
	\draw[edge,dotted] (b) to [bend right=30] (c) (a) -- (d) (e) -- (a)  (f) -- (b);

\end{tikzpicture}}
    \caption{}\label{fig:tuza-3trees-lemma-b}
  \end{subfigure}
  \begin{subfigure}[b]{.3\linewidth}
    \centering
    \scalebox{.7}{\begin{tikzpicture}[scale = 0.6]

	\draw[draw=none, fill=gray!60] (0,0) -- (90:3) -- (30:1.6) -- (0,0);
	\draw[draw=none, fill=gray!60] (0,0) -- (150:1.6) -- (210:3) -- (0,0);
	\draw[draw=none, fill=gray!60] (0,0) -- (-90:1.6) -- (-30:3) -- (0,0);
		
	\node (a) [black vertex] at (0,0) {};
	\node (b) [black vertex] at (90:3) {};
	\node (c) [black vertex] at (210:3) {};
	\node (d) [black vertex] at (-30:3) {};

	\node () [] at (-65:.5) {$a$};
	\node () [] at (90:3.6) {$b$};
	\node () [] at (210:3.4) {$c$};
	\node () [] at (-30:3.3) {$d$};

	\node (e) [black vertex] at (150:1.6) {};
	\draw[edge]  (e) -- (b)  (e) -- (c)  (e) -- (a);
	\node (f) [black vertex] at (30:1.6) {};
	\draw[edge] (f) -- (a)  (f) -- (d) (f) -- (b);
	\node (g) [black vertex] at (-90:1.6) {};
	\draw[edge] (g) -- (a) (g) -- (c) (g) -- (d);

	\draw[edge,dashed] (c) to [bend right=30] (d) (a) -- (b)  (b) to [bend right=30] (c) (a) -- (d)    (c) -- (a)   (d) to [bend right=30] (b);
\end{tikzpicture}}
    \caption{}\label{fig:tuza-3trees-lemma-c}
  \end{subfigure}
  \caption{Illustrations of the nodes of height $1$ of the proof of Claim \ref{claim:tuza-minimal-h(t)=1-then-S(t)=1}.}
  \label{fig:tuza-3trees-lemma}
\end{figure}
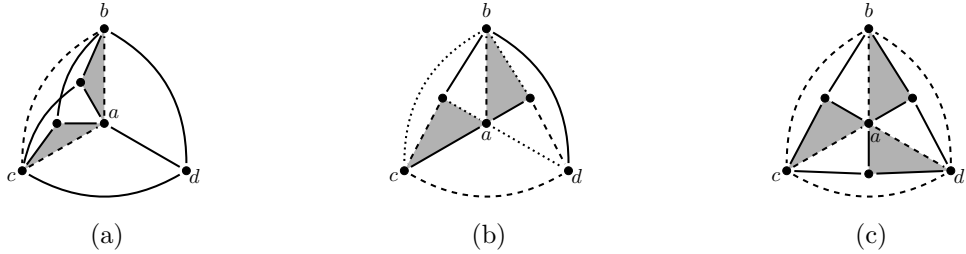

  Note that $h(r) > 2$.  
  Indeed, if $h(r) = 2$, then the only vertex $t$ in $S(r)$ is such that $h(t) = 1$, 
  and, by Claim~\ref{claim:tuza-minimal-h(t)=1-then-S(t)=1},~$|S(t)|=1$, 
  which would imply that $|V(T)| = 3$, a contradiction. 
  So $h(r) \geq 3$ and there is a node ${t \in V(T) \setminus \{r\}}$ such that~$h(t)=2$.
  Let~$L = \{\ell_1,\ldots,\ell_m\}$ be the set of successors of~$t$ that are leaves of~$T$,
  and let~$Q = S(t) \setminus L = \{q_1,\ldots,q_k\}$.
  Let~$u_i$ be the representative of~$\ell_i$ for~$i=1,\ldots,m$.
  By Claim~\ref{claim:tuza-minimal-h(t)=1-then-S(t)=1},
  $|S(q_i)| = 1$ for~$i=1,\ldots,k$.
  For every such~$i$, let~$S(q_i)=\{q'_i\}$, and let~$Q'=\{q'_1,\ldots, q'_k\}$.
  Let~$v_i$ be the representative of~$q_i$ 
  and~$v'_i$ be the representative of~$q'_i$ for~$i=1,\ldots,k$ (Figure~\ref{fig:names}).

  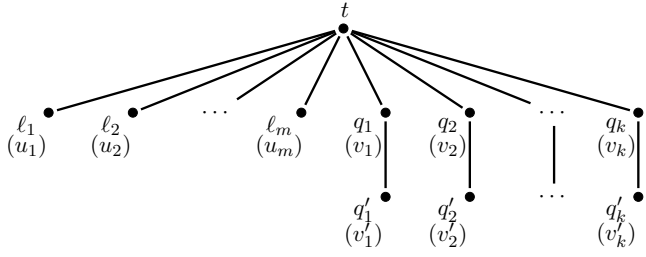
\begin{figure}[h]
    \centering
    \scalebox{0.8}{\begin{tikzpicture}[scale = .7]
	\node (t) [black vertex] 	at (0,1) {};
	\node (l1) [black vertex] 	at (-7,-1) {};
	\node (l2) [black vertex] 	at (-5,-1) {};
	\node (ldots) [] 		at (-3,-1) {$\cdots$};		
	\node (lm) [black vertex] 	at (-1,-1) {};
	\node (q1) [black vertex] 	at (1,-1) {};
	\node (q2) [black vertex] 	at (3,-1) {};
	\node (qdots) [] 		at (5,-1) {$\cdots$};
	\node (qk) [black vertex] 	at (7,-1) {};					
	\node (q1') [black vertex] 	at (1,-3) {};
	\node (q2') [black vertex] 	at (3,-3) {};
	\node (qdots') [] 		at (5,-3) {$\cdots$};
	\node (qk') [black vertex] 	at (7,-3) {};	
	
	\node () [] 	at ($(t)+(85:.45)$) {$t$};
	\node () [] 	at ($(l1)+(-150:.6)$) {$\ell_1$};
	\node () [] 	at ($(l2)+(-150:.6)$) {$\ell_2$};
	\node () [] 	at ($(lm)+(-150:.6)$) {$\ell_m$};
	\node () [] 	at ($(q1)+(-150:.6)$) {$q_1$};
	\node () [] 	at ($(q2)+(-150:.6)$) {$q_2$};
	\node () [] 	at ($(qk)+(-150:.6)$) {$q_k$};
	\node () [] 	at ($(q1')+(-150:.6)$) {$q_1'$};
	\node () [] 	at ($(q2')+(-150:.6)$) {$q_2'$};
	\node () [] 	at ($(qk')+(-150:.6)$) {$q_k'$};
	
	\node () [] 	at ($(l1)+(-150:.6)+(-90:.5)$) {$(u_1)$};
	\node () [] 	at ($(l2)+(-150:.6)+(-90:.5)$) {$(u_2)$};
	\node () [] 	at ($(lm)+(-150:.6)+(-90:.5)$) {$(u_m)$};
	\node () [] 	at ($(q1)+(-150:.6)+(-90:.5)$) {$(v_1)$};
	\node () [] 	at ($(q2)+(-150:.6)+(-90:.5)$) {$(v_2)$};
	\node () [] 	at ($(qk)+(-150:.6)+(-90:.5)$) {$(v_k)$};
	\node () [] 	at ($(q1')+(-150:.6)+(-90:.6)$) {$(v_1')$};
	\node () [] 	at ($(q2')+(-150:.6)+(-90:.6)$) {$(v_2')$};
	\node () [] 	at ($(qk')+(-150:.6)+(-90:.6)$) {$(v_k')$};	

        \draw[edge] (t) -- (l1);
        \draw[edge] (t) -- (l2);
        \draw[edge] (t) -- (ldots);
        \draw[edge] (t) -- (lm);
        \draw[edge] (t) -- (q1);
        \draw[edge] (t) -- (q2);
        \draw[edge] (t) -- (qdots);
        \draw[edge] (t) -- (qk);
        \draw[edge] (q1) -- (q1');
        \draw[edge] (q2) -- (q2');
        \draw[edge] (qdots) -- (qdots');
        \draw[edge] (qk) -- (qk');
\end{tikzpicture}}
    \caption{Notation for the descendants of $t$ in $T$ and their representatives.}
    \label{fig:names}
  \end{figure}

  Let~$V_t=\{a,b,c,d\}$ with~$V_t \cap V_{p(t)} = \{b,c,d\}$.

  \begin{claim} \label{claim:tuza-3tree-EDelta-subseteq-X'}
    Let~$\Delta$ be a triple of vertices in~$V_t$ with~$\Delta \neq bcd$.
    Let~$t' \in S^{\Delta}(t)$ and~$G'=G-R(t')$.
    If~$|S^{\Delta}(t)| \geq 3$, then
    there exists a minimum triangle transversal~$X'$ of~$G'$ with~$E(\Delta) \subseteq X'$.
  \end{claim}
  \begin{proof}
    Without loss of generality, $\Delta = abd$.
    Let~$X'$ be a minimum triangle transversal of~$G'$. 
    If~$E(\Delta) \subseteq X'$, then \(X'\) is the desired triangle transversal.
    So we may assume that~${E(\Delta) \not\subseteq X'}$ and, without loss of generality, 
    that ${S^{\Delta}(t) \setminus \{t'\}=\{\ell_1,\ldots,\ell_{m'}\} \cup \{q_1, \ldots, q_{k'}\}}$.
    Let~${F_i = \{au_i,bu_i,du_i\}}$ for~$i=1,\ldots,m'$;
    and~$H_i = E\big(G\big[V_{q_i} \cup V_{q'_i}\big]\big) \setminus E(abd)$ for~$i=1,\ldots,k'$.
    Since $E(\Delta) \not\subseteq X'$, we have that~$|X' \cap F_i|\geq 1$ for~$i=1,\ldots,m'$, 
    and~$|X' \cap H_i|\geq 2$ for~$i=1,\ldots,k'$.
    Therefore
    $X'' \ = \ \big(X' \setminus (\{F_i\colon i \in [m']\} \cup \{H_i\colon i \in [k']\})\big) 
                                               \cup E(abd) \cup \{v_iv'_i\colon i \in [k']\}$
    is a triangle transversal of~$G$, and
    $|X''| \ \leq \ |X'| - m' - 2k' + 3 - (|E(abd) \cap X'|) + k' 
             \ \leq \ |X'| - m' - k' + 2 \ \leq \ |X'|,$
    where the last inequality holds because~$m'+k'+1 = |S^{\Delta}(t)| \geq 3$.
    So \(X''\) is the desired transversal.
  \end{proof}

  \begin{claim}\label{claim:tuza-SDelta-cap-L-or-cap-L=empty}
    Let $\Delta$ be a triple of vertices in $V_t$ with $\Delta \neq bcd$.
    Then $|S^{\Delta}(t)| \leq 2$.
  \end{claim}

  \begin{proof}
    Without loss of generality,~${\Delta = abd}$.
    Suppose, for a contradiction, that~${|S^{\Delta}(t)| \geq 3}$.
    If there is a node in~${S^{\Delta}(t) \cap L}$, say~$\ell_1$, then let~${G'=G-u_1}$. 
    By Claim \ref{claim:tuza-3tree-EDelta-subseteq-X'},
    there is a minimum triangle transversal~$X'$ of~$G'$ with~$E(abd) \subseteq X'$. 
    By the minimality of~$G$, there is a triangle packing~$Y'$ in~$G'$ 
    such that $(X',Y')$ is a~$\frac{9}{5}$-TP of~$G'$.
    Then~$(X',Y')$ is also a~$\frac{9}{5}$-TP of~$G$, a contradiction.
    Similarly, if there is a node in~$S^{\Delta}(t) \cap Q$, say~$q_1$, then let~$G'=G-v_1-v'_1$.  
    By Claim \ref{claim:tuza-3tree-EDelta-subseteq-X'},
    there is a minimum triangle transversal~$X'$ of~$G'$ such that~$E(abd) \subseteq X'$. 
    By the minimality of~$G$, there is a triangle packing~$Y'$ in~$G'$ 
    such that $(X',Y')$ is a~$\frac{9}{5}$-TP of~$G'$.
    Now~$\big(X' \cup \{v_1v'_1\},Y' \cup \{v_1v'_1w_1\}\big)$ is a~$\frac{9}{5}$-TP of~$G$, 
    where~$w_1$ is a vertex adjacent to both~$v_1$ and~$v'_1$, a contradiction.
  \end{proof}

  Given a 4-clique~$K$ in~$G$ and a set~${A \subseteq E(G)}$ such that~${E(K) \cap A = \{e\}}$, 
  we denote by~${K \otimes A}$ the only edge in~$K$ 
  that does not share any vertex with~$e$. For instance,~${V_t \otimes \{bc\}=ad}$.
  
  \begin{claim}\label{claim:3-trees-transversals} 
    Let~${G'=G-R(t)}$.
    If~$X'$ is a triangle transversal of~$G'$, 
    then there exists a triangle transversal of~$G$ with at most size~$|X'|+\min\{5+k,1+m+2k\}$.
  \end{claim}

  \begin{proof}
    From $X'$, we will build two triangle transversals of~$G$,
    one of size at most \({|X'|+5+k}\), and another one of size at most \({|X'|+1+m+2k}\).
    Clearly ${X' \cap E(bcd) \neq \emptyset}$, 
    because~$X'$ is a triangle transversal of~$G'$ and~$bcd$ is a triangle in~$G'$.
    Hence~${|X' \cap E\big(G[V_t]\big)| \geq 1}$ 
    and therefore~$X_1 = X' \cup E\big(G[V_t]\big) \cup \{v_iv'_i\colon i \in [k]\}$ is 
    a triangle transversal of~$G$ with 
    ${|X_1| = |X'|+|E\big(G[V_t]\big)| - |X' \cap E\big(G[V_t]\big)|+k \leq |X'|+5+k}$.
    For the second transversal, without loss of generality, assume that~${bc \in X'}$.
    Let~${e_i = V_{\ell_i} \otimes \{bc,ad\}}$ for~${i=1,\ldots,m}$, and 
    let ${f_i = V_{q_i} \otimes \{bc,ad\}}$ and~${f'_i= V_{q'_i} \otimes \{bc,ad,f_i\}}$ for~${i=1,\ldots,k}$.
    Then the second triangle transversal of $G$ is 
    $X_2 = X' \cup \{bc,ad\} \cup {\{e_1,\ldots,e_m\}}
    \cup {\{f_1,\ldots,f_k\}} \cup {\{f'_1,\ldots,f'_k\}}$, 
    for which ${|X_2|=|X'|+1+m+2k}$. 
  \end{proof}

  Now, we may conclude the proof of the theorem.
  For every triple of vertices~$\Delta \subseteq V_t$ such that~$\Delta \neq bcd$,
  let~${h(\Delta)=1+\max\big\{h(t')\colon t' \in S^{\Delta}(t)\big\}}$,
  let \(k' = \big|\big\{\Delta \subseteq V_t\colon h(\Delta)=2\big\}\big|\)
  and \({m' = \big|\big\{\Delta \subseteq V_t\colon h(\Delta)=1\big\}\big|}\).
  Note that~\(k' \geq 1\) because $h(t) = 2$.  Also \(k'+m' \leq 3\).
  Let \(G^+\) be the graph obtained from \(G\big[V_t\cup R(t)\big]\) by 
  removing, for each triple of vertices~$\Delta \subseteq V_t$
  such that $\Delta \neq bcd$ and \(|S^\Delta(t)|=2\), 
  the vertices in \(R(x)\), for precisely one \(x\in S^\Delta(t)\)
  with $x \in L$ whenever possible.
  By Claim~\ref{claim:tuza-SDelta-cap-L-or-cap-L=empty}, 
  graph \(G^+\) is a restricted planar \(3\)-tree.
  Also, \(G^+\) has \(f = 4+4k'+2m'\) faces.
  By Proposition~\ref{prop:restricted-3-tree}, 
  \(G^+\) contains a triangle packing~\(\calP^+\) of facial triangles
  containing~\(bcd\) and such that \(|\calP^+|\geq\lceil f/3\rceil\),
  and if~\(G^+\) is super restricted, then~\(|\calP^+|=5\).
  Let \(Q^+\) be the set \(\big\{q_i\in Q \colon v_i\notin V(G^+)\big\}\).
  For each~\(q_i\in Q^+\), let~\(T_i = v_iv_i'w_i\), 
  where \(w_i\) is a vertex adjacent to both~\(v_i\) and~\(v_i'\),
  and let \(\calP = \calP^+\cup\{T_i \colon q_i\in Q^+\}\).
  Note that~\(|\calP| = |\calP^+| + (k-k')\).

  Let \(G'=G-R(t)\) and \((X',Y')\) be a \(\frac{9}{5}\)-TP of \(G'\).
  The only triangle in~\(\calP\) containing edges of~\(G'\) is~\(bcd\).
  Hence \(Y = Y'\cup \big(\calP\setminus\{bcd\}\big)\) is a triangle packing of \(G\)
  of size at least \(|Y'|+|\calP^+| - 1 + (k-k')\).
  By Claim~\ref{claim:3-trees-transversals}, 
  there exists a triangle transversal~$X$ of~$G$ with size at most $|X'|+\min\{5+k,1+m+2k\}$.

  If \(m' = 2\), then \(k'=1\) and \(G^+\) is super restricted.
  In this case, $|\calP^+| = 5$, and therefore
  \({|Y| \geq |Y'|+|\calP^+|-1 + (k-k') = |Y'| + 5 - 1 + k - 1 = |Y'| + 3 + k}\).
  Also, \({5(5+k) < 9(3+k)}\), thus
  $$ |X| \ \leq \ |X'| + 5+k \ < \ \frac{9}{5}(|Y'|+3+k) +\frac{1}{5} 
         \ \leq \ \frac95|Y|+\frac15, $$
  and $(X,Y)$ would be a $\frac95$-TP of $G$, a contradiction. 
  So \(m' \leq 1\).

  Now, a simple calculation shows that, 
  if \(k\in\{1,2,3,4,5,6\}\), \(k'\in\{1,2,3\}\), \(m\in\{0,1,2,3,4\}\), 
  and \(m'\in\{0,1\}\) are such that \(k\geq k'\geq \lceil k/2\rceil\), 
  \(m\geq m'\geq \lceil m/2\rceil\), and \(m'+k'\leq 3\),
  then we have 
  \[
    \min\{5+k,1+m+2k\}\leq \frac{9}{5}\left(\left\lceil \frac{f}{3}\right\rceil-1+(k-k')\right).
  \]
  Therefore, since \(|\mathcal{P}^+|\geq \lceil f/3\rceil\), we have
  \begin{eqnarray*}
    |X| &\leq&	|X'|+\min\{5+k,1+m+2k\} \\
        &\leq&	\frac{9}{5}|Y'| + \frac{1}{5} 
               + \frac{9}{5}\left(\left\lceil \frac{f}{3}\right\rceil-1+(k-k')\right) \\
        &\leq&	\frac{9}{5}\left(|Y'| + \left\lceil \frac{f}{3}\right\rceil-1+(k-k')\right) + \frac{1}{5} \\
        &\leq&	\frac{9}{5}|Y| + \frac{1}{5}.
  \end{eqnarray*}
  implying that $(X,Y)$ is a $\frac95$-TP of $G$, a contradiction. 
  This concludes the proof. 
\end{proof}

\section{Concluding remarks} \label{section:conclusion}

In this paper we present three results related to Tuza's Conjecture.
In Section~\ref{section:tuza-partial-6-trees}, we obtained a lemma (Lemma~\ref{lemma:tuza-tw6-main}) that extends Puleo's tools~\cite{Puleo15},
and allowed us to verify Tuza's Conjecture for graphs with treewidth at most \(6\).
Any minimal counterexample to Tuza's Conjecture is an irreducible robust graph, 
so Lemma~\ref{lemma:tuza-tw6-main} might help in achieving further results regarding this problem.
In Sections~\ref{section:tuza-planar-triangulation} and~\ref{section:tuza-3-trees}, 
we obtained stronger versions of Tuza's Conjecture for specific classes of graphs.
We believe that the techniques used here may also be useful to deal with other classes of graphs,
perhaps by introducing new ingredients.

\section*{Acknowlegments}

We would like to thank the reviewer for valuable comments.

\smallskip\noindent
This study was financed in part by the Coordenação de Aperfeiçoamento 
de Pessoal de Nível Superior - Brasil (CAPES) - Finance Code 001,
by Conselho Nacional de Desenvolvimento Científico e Tecnológico -- CNPq (Proc~456792/2014-7, 308116/2016-0,~423395/2018-1, and 423833/2018-9),
by Fundação de Amparo à Pesquisa do Estado do Rio de Janeiro -- FAPERJ (Proc.~211.305/2019),
by Fundação de Amparo à Pesquisa do Estado de São Paulo -- FAPESP (Proc.~2013/03447-6 and 2015/08538-5), and by Project MaCLinC of NUMEC/USP.

\end{document}